\documentclass[11pt,dvipsnames,reqno]{amsart}

  \usepackage{amssymb,mathtools,amsthm}

  \usepackage[a4paper,left=3cm,right=3cm,top=3cm,bottom=3cm]{geometry}
  \usepackage[T1]{fontenc}
\usepackage{microtype}
\newcommand\numberthis{\addtocounter{equation}{1}\tag{\theequation}}
\numberwithin{equation}{section}
  \theoremstyle{plain}
    \newtheorem{lemma}[equation]{Lemma}
    \newtheorem{theorem}[equation]{Theorem}
    \newtheorem{proposition}[equation]{Proposition}
    
  \theoremstyle{definition}
    \newtheorem{definition}[equation]{Definition}
    \newtheorem{remark}[equation]{Remark}

  \renewcommand{\complement}{\mathsf{c}}

  \usepackage{caption}
  \usepackage{afterpage}
  \usepackage{tikz}
    \usetikzlibrary{decorations.markings}
      \tikzset{->-/.style={decoration={markings,mark=at position #1 with {\arrow{>}}},postaction={decorate}}}
  \usepackage{enumitem}

  \usepackage{hyperref}
\definecolor{colorlinks}{RGB}{0, 24, 168}
\definecolor{colorcites}{RGB}{124, 10, 2}
\hypersetup{
    colorlinks=true,
    linkcolor=colorlinks,
    citecolor=colorcites,
    urlcolor=colorlinks,
    pdfborder={0 0 0}
}

  \title[A generalisation of the dimer model]{A generalisation of the honeycomb\\dimer model to higher dimensions}
  \author{Piet Lammers}
  \date{}
  \address{Statistical Laboratory, Centre for Mathematical Sciences, University of Cambridge}
  \email{p.g.lammers@statslab.cam.ac.uk}
\usepackage{todonotes}

\newcommand\x{\mathbf x}
\newcommand\y{\mathbf y}
\newcommand\n{\mathbf n}
\newcommand\e{\mathbf e}
\newcommand\g{\mathbf g}
\newcommand\0{\mathbf 0}
\newcommand\s{\mathbf s}
\newcommand\z{\mathbf z}
\newcommand\R{\mathbb R}
\newcommand\Z{\mathbb Z}
\renewcommand\P{\mathbb P}
\newcommand\E{\mathbb E}
\newcommand\p{\mathbf p}

\begin{document}

% \todo{---}
% \todo{$\smallsetminus\smallsetminus$}

\makeatletter
\@namedef{subjclassname@2020}{\textup{2020} Mathematics Subject Classification}
\makeatother
%!TEX root= ../main.tex
\begin{abstract}
    Linde, Moore, and Nordahl introduced a generalisation of the honeycomb dimer model to higher dimensions.
    The purpose of this article is to describe a number of structural properties of this generalised model.
    First, it is shown that the samples of the model are in one-to-one correspondence with the perfect matchings of a hypergraph.
    This leads to a generalised Kasteleyn theory: the partition function of the model equals the Cayley hyperdeterminant of the adjacency hypermatrix of the hypergraph.
    Second, we prove an identity which relates the covariance matrix of the random height function directly to the random geometrical structure of the model.
    This identity is known in the planar case but is new for higher dimensions.
    It relies on a more explicit formulation of Sheffield's \emph{cluster swap} which is made possible by the structure of the honeycomb dimer model.
    Finally, we use the special properties of this explicit cluster swap to give a new and simplified proof of strict convexity of the surface tension in this case.
\end{abstract}

\maketitle

\setcounter{tocdepth}{1}
\tableofcontents

%!TEX root= ../main.tex

\section{Introduction}
\label{sec:intro}

\subsection{Background}
Random models on shift-invariant Euclidean graphs such as the square lattice
and the hexagonal lattice form a well-known subject of study in both combinatorics and statistical physics~\cite{grimmett_2018}.
There are several integrable models which allow for a quantitative analysis.
In the integrable setting, the focus is on deriving quantitative results concerning (asymptotics of)
 partition functions and correlation functions.
Examples of such models are
the Ising model~\cite{ONSAGER,YANG_Ising},
ice-type models~\cite{Lieb,Sutherland,Yang_ice},
and the dimer model~\cite{KASTELEYN_ORIGINAL,TEMPERLEY_FISHER_ORIGINAL},
see also~\cite{BAXTER}.
These quantitative estimates in turn imply qualitative results,
 for example (non)uniqueness of shift-invariant Gibbs measures,
and (when the samples are height functions) strict convexity of
the surface tension.
Sometimes it is possible to derive qualitative results even in the absence
of quantitative estimates.
Georgii~\cite{GEORGII} provides an excellent overview for the theory of Gibbs
measures for
general (non-integrable) models,
and Sheffield~\cite{SHEFFIELD} derives many key results
for models of height functions in the gradient setting,
including strict convexity of the surface tension in any dimension.

The focus of this paper is a natural generalisation of the hexagonal dimer model (Figure~\ref{fig:2d}\ref{fig:2d:hex}) to arbitrary dimension.
The generalised model
first appeared in the work of
 Linde, Moore, and Nordahl~\cite{Linde}. It also
belongs to the category of models under consideration in the thesis of Sheffield~\cite{SHEFFIELD}.
A great deal is known about the original two-dimensional dimer model: we mention three pivotal developments.
The mathematical study of dimers was initiated by
Kasteleyn, Temperley, and Fisher.
Kasteleyn~\cite{KASTELEYN_ORIGINAL} and Temperley and Fisher~\cite{TEMPERLEY_FISHER_ORIGINAL}
independently calculated the number of perfect matchings of an $n\times m$ grid or,
equivalently, the number of domino tilings of an $n\times m$ rectangle.
Kasteleyn~\cite{KASTELEYN_ORIGINAL_More_graphs,KASTELEYN_in_this_other_book}
later showed that the number of perfect matchings of any
bipartite planar graph equals the determinant of a matrix
that is closely related to the adjacency matrix of the concerned graph.
Cohn, Kenyon, and Propp~\cite{COHN} proved the variational
principle for domino tilings; the scope of their article includes the hexagonal dimer model.
Remarkably, a closed-form solution
for the surface tension is found, something that is not to be expected in higher dimensions.
Their derivation of the closed-form formula relies
on a bijection between dimer configurations and height functions,
together with an
 original application of the Kasteleyn theory.
 Kenyon, Okounkov, and Sheffield~\cite{amoebae} establish a bijection
between the set of accessible slopes and the set of ergodic Gibbs measures.
They furthermore classify the ergodic Gibbs measures into three categories
(frozen, liquid, and gaseous)
which describe qualitatively the behaviour of the random surface.
Their paper contains many more qualitative and quantitative results.

% !TEX root =  ../main.tex

\begin{figure}
\captionsetup{singlelinecheck=off}
\begin{center}
\newcommand{\yheight}{0.866025}%=sqrt(3/4)
\newcommand{\dax}{0.288675}%=tan(30)/2
\begin{tikzpicture}[x=1cm,y=1cm,scale=12/20.8]

  \begin{scope}
  \node[anchor=north] at (\yheight*5,-0.9) {(a)};
  \begin{scope}
    \clip (\yheight*1.1,-0.9)  rectangle (\yheight*8.9,6.9);
    \newcommand{\loza}[2]{
      \draw[very thin,fill=NavyBlue!10]
        (2*\yheight*#1+\yheight*#2,#2*0.5)--
        (\yheight*2*#1+\yheight+\yheight*#2,#2*0.5-0.5)--
        (\yheight*2*#1+\yheight*2+\yheight*#2,#2*0.5)--
        (\yheight*2*#1+\yheight*#2+\yheight*1,#2*0.5+0.5)--
        cycle;
    }
    \newcommand{\lozb}[2]{
      \draw[very thin,fill=NavyBlue!50]
        (\yheight*2*#1+\yheight*#2+\yheight*1,#2*0.5+0.5)--
        (\yheight*2*#1+\yheight*#2+\yheight*2,#2*0.5)--
        (\yheight*2*#1+\yheight*#2+\yheight*2,#2*0.5+1)--
        (\yheight*2*#1+\yheight*#2+\yheight*1,#2*0.5+1.5)--
        cycle;
    }
    \newcommand{\lozc}[2]{
      \draw[very thin,fill=NavyBlue]
        (\yheight*2*#1+\yheight*#2+\yheight*1,#2*0.5-0.5)--
        (\yheight*2*#1+\yheight*#2+\yheight*2,#2*0.5)--
        (\yheight*2*#1+\yheight*#2+\yheight*2,#2*0.5-1)--
        (\yheight*2*#1+\yheight*#2+\yheight*1,#2*0.5-1.5)--
        cycle;
    }
    \input{figures/2d_data_lozenge0.tex}
  \end{scope}
  \end{scope}

  \begin{scope}[shift={(7.5,0)}]
  \node[anchor=north] at (\yheight*5,-0.9) {(b)};
  \begin{scope}
    \clip (\yheight*1.1,-0.9)  rectangle (\yheight*8.9,6.9);
    \newcommand{\loza}[2]{
      \draw[very thin,gray]
        (2*\yheight*#1+\yheight*#2,#2*0.5)--
        (\yheight*2*#1+\yheight+\yheight*#2,#2*0.5-0.5)--
        (\yheight*2*#1+\yheight*2+\yheight*#2,#2*0.5)--
        (\yheight*2*#1+\yheight*#2+\yheight*1,#2*0.5+0.5)--
        cycle;
      \draw[very thin,dotted,gray]
        (\yheight*2*#1+\yheight*#2+\yheight*1,#2*0.5+0.5)--
        (\yheight*2*#1+\yheight*1+\yheight*#2,#2*0.5-0.5);
    }
    \newcommand{\lozb}[2]{
      \draw[very thin,gray]
        (\yheight*2*#1+\yheight*#2+\yheight*1,#2*0.5+0.5)--
        (\yheight*2*#1+\yheight*#2+\yheight*2,#2*0.5)--
        (\yheight*2*#1+\yheight*#2+\yheight*2,#2*0.5+1)--
        (\yheight*2*#1+\yheight*#2+\yheight*1,#2*0.5+1.5)--
        cycle;
      \draw[very thin,dotted,gray]
        (\yheight*2*#1+\yheight*#2+\yheight*1,#2*0.5+0.5)--
        (\yheight*2*#1+\yheight*#2+\yheight*2,#2*0.5+1);
    }
    \newcommand{\lozc}[2]{
      \draw[very thin,gray]
        (\yheight*2*#1+\yheight*#2+\yheight*1,#2*0.5-0.5)--
        (\yheight*2*#1+\yheight*#2+\yheight*2,#2*0.5)--
        (\yheight*2*#1+\yheight*#2+\yheight*2,#2*0.5-1)--
        (\yheight*2*#1+\yheight*#2+\yheight*1,#2*0.5-1.5)--
        cycle;
      \draw[very thin,dotted,gray]
        (\yheight*2*#1+\yheight*#2+\yheight*1,#2*0.5-0.5)--
        (\yheight*2*#1+\yheight*#2+\yheight*2,#2*0.5-1);
    }
    \newcommand{\cipher}[3]{
      \node at (\yheight*2*#1+\yheight*#2,#2*0.5) {#3};
    }
    \input{figures/2d_data_lozenge0.tex}
    \input{figures/2d_data_height0.tex}
  \end{scope}
  \end{scope}

  \begin{scope}[shift={(15,0)}]
  \node[anchor=north] at (\yheight*5,-0.9) {(c)};
  \begin{scope}
    \clip (\yheight*1.1,-0.9)  rectangle (\yheight*8.9,6.9);
    \newcommand{\fatdot}[2]{
      \fill (#1,#2) circle (0.05*20.8/12);
    }
    \newcommand{\loza}[2]{
      \draw[very thin,gray]
        (2*\yheight*#1+\yheight*#2,#2*0.5)--
        (\yheight*2*#1+\yheight+\yheight*#2,#2*0.5-0.5)--
        (\yheight*2*#1+\yheight*2+\yheight*#2,#2*0.5)--
        (\yheight*2*#1+\yheight*#2+\yheight*1,#2*0.5+0.5)--
        cycle;
      \draw[very thick]
        (\yheight*2*#1+\yheight+\yheight*#2-\dax,#2*0.5)--
        (\yheight*2*#1+\yheight*1+\yheight*#2+\dax,#2*0.5);
      \draw[very thick,dotted]
        (\yheight*2*#1+\yheight+\yheight*#2-\dax,#2*0.5)--
        (2*\yheight*#1+\yheight*#2+0.5*\yheight,#2*0.5+0.25);
      \draw[very thick,dotted]
        (\yheight*2*#1+\yheight+\yheight*#2-\dax,#2*0.5)--
        (2*\yheight*#1+\yheight*#2+0.5*\yheight,#2*0.5-0.25);
      \draw[very thick,dotted]
        (\yheight*2*#1+\yheight+\yheight*#2+\dax,#2*0.5)--
        (2*\yheight*#1+\yheight*#2+1.5*\yheight,#2*0.5+0.25);
      \draw[very thick,dotted]
        (\yheight*2*#1+\yheight+\yheight*#2+\dax,#2*0.5)--
        (2*\yheight*#1+\yheight*#2+1.5*\yheight,#2*0.5-0.25);
      \fill(\yheight*2*#1+\yheight+\yheight*#2-\dax,#2*0.5) circle (0.05*20.8/12);
      \fill(\yheight*2*#1+\yheight+\yheight*#2+\dax,#2*0.5) circle (0.05*20.8/12);
      \fill[color=white](\yheight*2*#1+\yheight+\yheight*#2+\dax,#2*0.5) circle (0.04*20.8/12);
    }
    \newcommand{\lozb}[2]{
      \draw[very thin,gray]
        (\yheight*2*#1+\yheight*#2+\yheight*1,#2*0.5+0.5)--
        (\yheight*2*#1+\yheight*#2+\yheight*2,#2*0.5)--
        (\yheight*2*#1+\yheight*#2+\yheight*2,#2*0.5+1)--
        (\yheight*2*#1+\yheight*#2+\yheight*1,#2*0.5+1.5)--
        cycle;
      \draw[very thick]
        (\yheight*2*#1+2*\yheight+\yheight*#2-\dax,#2*0.5+0.5)--
        (\yheight*2*#1+\yheight+\yheight*#2+\dax,#2*0.5+1);
      \draw[very thick,dotted]
        (\yheight*2*#1+2*\yheight+\yheight*#2-\dax,#2*0.5+0.5)--
        (2*\yheight*#1+\yheight*#2+1.5*\yheight,#2*0.5+0.25);
      \draw[very thick,dotted]
        (\yheight*2*#1+2*\yheight+\yheight*#2-\dax,#2*0.5+0.5)--
        (\yheight*2*#1+2*\yheight+\yheight*#2,#2*0.5+0.5);
      \draw[very thick,dotted]
        (\yheight*2*#1+\yheight+\yheight*#2+\dax,#2*0.5+1)--
        (\yheight*2*#1+\yheight+\yheight*#2,#2*0.5+1);
      \draw[very thick,dotted]
        (\yheight*2*#1+\yheight+\yheight*#2+\dax,#2*0.5+1)--
        (\yheight*2*#1+1.5*\yheight+\yheight*#2,#2*0.5+1.25);
      \fill(\yheight*2*#1+2*\yheight+\yheight*#2-\dax,#2*0.5+0.5) circle (0.05*20.8/12);
      \fill(\yheight*2*#1+\yheight+\yheight*#2+\dax,#2*0.5+1) circle (0.05*20.8/12);
      \fill[color=white](\yheight*2*#1+\yheight+\yheight*#2+\dax,#2*0.5+1) circle (0.04*20.8/12);
    }
    \newcommand{\lozc}[2]{
      \draw[very thin,gray]
        (\yheight*2*#1+\yheight*#2+\yheight*1,#2*0.5-0.5)--
        (\yheight*2*#1+\yheight*#2+\yheight*2,#2*0.5)--
        (\yheight*2*#1+\yheight*#2+\yheight*2,#2*0.5-1)--
        (\yheight*2*#1+\yheight*#2+\yheight*1,#2*0.5-1.5)--
        cycle;
      \draw[very thick]
        (\yheight*2*#1+2*\yheight+\yheight*#2-\dax,#2*0.5-0.5)--
        (\yheight*2*#1+\yheight+\yheight*#2+\dax,#2*0.5-1);
      \draw[very thick,dotted]
        (\yheight*2*#1+2*\yheight+\yheight*#2-\dax,#2*0.5-0.5)--
        (2*\yheight*#1+\yheight*#2+1.5*\yheight,#2*0.5-0.25);
      \draw[very thick,dotted]
        (\yheight*2*#1+2*\yheight+\yheight*#2-\dax,#2*0.5-0.5)--
        (\yheight*2*#1+2*\yheight+\yheight*#2,#2*0.5-0.5);
      \draw[very thick,dotted]
        (\yheight*2*#1+\yheight+\yheight*#2+\dax,#2*0.5-1)--
        (\yheight*2*#1+\yheight+\yheight*#2,#2*0.5-1);
      \draw[very thick,dotted]
        (\yheight*2*#1+\yheight+\yheight*#2+\dax,#2*0.5-1)--
        (\yheight*2*#1+1.5*\yheight+\yheight*#2,#2*0.5-1.25);
      \fill(\yheight*2*#1+2*\yheight+\yheight*#2-\dax,#2*0.5-0.5) circle (0.05*20.8/12);
      \fill(\yheight*2*#1+\yheight+\yheight*#2+\dax,#2*0.5-1) circle (0.05*20.8/12);
      \fill[color=white](\yheight*2*#1+\yheight+\yheight*#2+\dax,#2*0.5-1) circle (0.04*20.8/12);
    }
    \input{figures/2d_data_lozenge0.tex}
  \end{scope}
  \end{scope}
\end{tikzpicture}
\vspace{1em}
\caption[Dimension $d=2$]{Dimension $d=2$; projections onto $H$ of several representations:
\begin{enumerate}[label=(\alph*),ref=\alph*]
  \item\label{fig:2d:geo} As a stepped surface or lozenge tiling,
  \item\label{fig:2d:hf} As a height function on the simplicial lattice $(X^d,E^d)$,
  \item\label{fig:2d:hex} As a dimer cover or perfect matching of the dual hexagonal lattice.
\end{enumerate}
}
\label{fig:2d}
\end{center}
\end{figure}

\subsection{The honeycomb dimer model in dimension $d\geq 2$}
%Let us now formally introduce the model of interest,
%and give an overview of what is known.
Let $(X^d,E^d)$ denote the graph obtained from the square lattice $\mathbb Z^{d+1}$
by identifying vertices which differ by an integer multiple of the vector
$\mathbf n:=\mathbf e_1+\dots+\mathbf e_{d+1}$.
This graph is called the \emph{simplicial lattice};
its vertices are equivalence classes of vertices of the square lattice.
Let $\Omega$ denote the set of functions $f:X^d\to\mathbb Z$
which have the property that $f(\0)\in(d+1)\mathbb Z$
and
$
f([\x+\e_i])-f([\x])\in \{-d,1\}
$
for any $\x\in\mathbb Z^{d+1}$ and $1\leq i\leq d+1$.
Functions in $\Omega$ are called \emph{height functions}---see Figure~\ref{fig:2d}\ref{fig:2d:hf}.
The set $\Omega$ is in bijection with the set of \emph{stepped surfaces}
in $\mathbb R^{d+1}$.
Informally, a stepped surface is a union of unit hypercubes with
integer coordinates, such that each hypercube is well-supported, and such that there is no overhang.
Stepped surfaces are related directly
to the three-dimensional interpretation of the familiar picture of lozenge tilings for $d=2$,
see Figure~\ref{fig:2d}\ref{fig:2d:geo}.
Each stepped surface is furthermore associated with a \emph{tiling}
of the hyperplane orthogonal to $\n$;
this tiling is essentially obtained by projecting the exposed
faces of the hypercubes of the stepped surface onto this hyperplane.
These bijections are all introduced in the work of
Linde, Moore, and Nordahl~\cite{Linde}.

In this article we are interested in the model of uniformly random height functions
whenever fixed or periodic boundary conditions are enforced.
For fixed boundary conditions, it is also possible to consider
more general Boltzmann measures, in the spirit of the classical planar dimer model.
The purpose of this article is to point out a number of structural properties of the generalised model
which lead to new results.

\subsection{The double dimer model and the cluster swap}
For the double dimer model, one superimposes two dimer configurations of the same graph.
The union of two such dimer covers decomposes into a number of isolated
edges which appear in both dimer covers, called \emph{double edges},
and a number of closed loops of even length, where each edge of the loop
is contained in exactly one of the original dimer covers.
If the distribution of the two dimer configurations is uniformly random,
subject say to fixed boundary conditions,
then the orientation of each loop is uniformly random in its two states.
More precisely, this means that for each loop, one can flip a fair coin
to decide if each dimer should change the configuration that it belongs to,
without changing the distribution of the product measure.
Many results have been obtained for the double dimer model:
see the work of Kenyon and Pemantle~\cite{Kenyon2014,KENYON201627} for the relevant literature.

Sheffield introduces \emph{cluster swapping}
in the seminal monograph \emph{Random surfaces}~\cite{SHEFFIELD}.
The technique employs the same idea of considering two configurations---height
functions, in this case---at once, then identifying and resampling independent
structures.
The cluster swap applies to \emph{simply attractive potentials},
that is, models of height functions which are induced by
a convex nearest-neighbour potential.
The setup is much more intricate than for the original double dimer model.
To compensate for the possible change in potential,
one first identifies a ferromagnetic Ising model,
then achieves independence through the
Edwards-Sokal representation~\cite{ES}
which in turn derives from the Swendsen-Wang update~\cite{SW}.
The cluster swap applies directly to the heights of the two height
functions; it is not an operation on their gradients.

\subsection{Main results}
First, we develop a new construction, the \emph{cluster boundary swap}, for the generalised model.
This cluster boundary swap is entirely analogous to the resampling operation
in the double dimer model.
In particular, the difference of two uniformly random height functions
decomposes as a geometrical structure consisting of
boundaries---which generalise loops---and double edges,
and, conditional on this structure, a number of fair coin flips, one for the orientation of each boundary.
Moreover, the operation directly manipulates the gradients of the two height functions,
which works to our advantage.
The name \emph{cluster boundary swap} is chosen intentionally,
as it should be considered a special case of the cluster swap,
adapted and optimised for the special geometrical structure that
is present in the generalisation of the honeycomb dimer model to higher dimensions.
The author is not aware of a similar construction for any other model in dimension
$d>2$; in particular, no generalisation is known for the dimer model
on the two-dimensional square lattice.

Second, we use the cluster boundary swap to obtain an identity
which relates the covariance matrix of the random height function $f$
to the geometrical structure of the model.
We prove that the variance of $f$ at a vertex $\x\in X^d$
is exactly $\frac12 (d+1)^2$ times the expectation
of the number of boundaries separating $\x$ from infinity in the product measure.
The identity makes sense for fixed boundary conditions only,
but it does apply to general Boltzmann measures
(which includes the uniform probability measure).
Similarly, for $\x,\y\in X^d$, we prove that the covariance between $f$
at $\x$ and $\y$ equals $\frac12 (d+1)^2$
times the number of boundaries that separate both $\x$
and $\y$ from infinity.
This identity is new and was previously known, to our knowledge, only for two-dimensional
dimer models.

%First, we use the cluster boundary swap to obtain a surprising identity
%for the covariance structure of the random height function $f$,
%after applying fixed boundary conditions.
%The theory applies both to the uniform measure, as well as to general Boltzmann measures.
%For the double dimer model, the variance of the height function
%at some vertex $\x$
%is exactly $\frac92$ times the expectation of the number of loops surrounding $\x$.
%The covariance between the height at two vertices $\x$ and $\y$ equals $\frac92$
%times the expected number of loops that surround both $\x$ and $\y$.
%The results for the generalised model are entirely analogous:
%we identify boundary surfaces which are comprised of the edges on which
%the two gradients $\nabla f_1$ and $\nabla f_2$ disagree,
%and the covariance of $f_1(\x)$ and $f_1(\y)$
%is exactly $\frac12 (d+1)^2$ times the number of boundary surfaces that
%surround both $\x$ and $\y$ in the product of the original Boltzmann measure with itself.
%This is a new result, see Theorems~\ref{thm:varianceboundofficial}
%and~\ref{thm:covarianceboundofficial}.

Third, we prove that the set of tilings, as introduced in~\cite{Linde},
are in one-to-one correspondence with the perfect matchings of a hypergraph
which is the natural dual of the simplicial lattice $(X^d,E^d)$.
This hypergraph is $d!$-regular and $d!$-partite,
in the sense that there is a partition of its vertex set
into $d!$-parts, such that each hyperedge contains exactly one
vertex in each part.
We derive a generalised Kasteleyn theory:
we prove that the partition function of any Boltzmann measure
equals the Cayley hyperdeterminant of the adjacency hypermatrix of
this hypergraph.
%This is a new result.
%See Section~\ref{sec:Kasteleyn}, in particular Theorem~\ref{thm:KAST}.
It is well-known that no Kasteleyn weighting
is necessary for the Kasteleyn theory on the hexagonal lattice.
We prove the same here: we may take the adjacency hypermatrix
without modifications as our Kasteleyn hypermatrix.
Unfortunately, this does not by the knowledge of the author
lead to an explicit formula for the surface tension,
nor to an expression for dimer-dimer correlations in terms of an inverse of the Kasteleyn hypermatrix.

Finally, we use the special properties of the cluster
boundary swap---relative to the more general cluster swap---to
greatly streamline the proof in~\cite{SHEFFIELD}
for strict convexity of the surface tension.
Note that strict convexity of the surface tension for $d=2$ was first derived
in~\cite{COHN} by means of an explicit calculation.
Strict convexity of the surface tension is important,
because it implies that the model is stable on a macroscopic scale.
The macroscopic behaviour of the model is described by a large devations principle,
which in turn implies a variational principle---these are due to generic arguments
which do not require that the surface tension is strictly convex.
If the surface tension is strictly convex, however,
then the rate function in the large deviations principle
has a unique minimizer, and consequently the random height function
concentrates around a unique limit shape.
See~\cite{COHN,MR2358053,kenyonlimit} for the variational principle in the original dimer setting.

\subsection{Structure of the article}

Section~\ref{sec:formal_intro} formally introduces the model,
and gives an overview of the several representations of each sample.
These constructions derive directly from~\cite{Linde}.
Section~\ref{sec:random_height_functions} describes the probability
measures under consideration, by enforcing fixed or periodic boundary conditions.
We construct the cluster boundary swap in Section~\ref{sec:CBS}.
In Section~\ref{sec:varcovar}, we state and prove the identity for the covariance structure
of the random height function,
and Section~\ref{sec:Kasteleyn} contains the generalised Kasteleyn theory.
The purpose of the remainder of the article
is to motivate and prove the result of strict convexity of the surface tension.
In Section~\ref{sec:gibbs} we introduce gradient Gibbs measures,
which play an important role in the proof,
but which are also interesting in their own right.
Section~\ref{sec:Statement of the variational principle}
motivates the result:
it introduces the surface tension,
and describes how it is related to the large deviations principle and the variational
principle.
Strict convexity of the surface tension is finally proven in Section~\ref{sec:strict}.

%\subsection{Open problems}

%!TEX root= main.tex

\section{Stepped surfaces, tilings, height functions}
\label{sec:formal_intro}

Linde, Moore, and Nordahl~\cite{Linde} observed that each sample of the model
has three natural representations.
The purpose of this section is to give an overview of these representations,
and to state or derive some basic properties.
For details, refer to~\cite{Linde}.
Throughout this article, $d\geq 2$ denotes the fixed dimension that we work in.

\subsection{Stepped surfaces}
Informally, a stepped surface is a union of unit hypercubes with corners in $\mathbb Z^{d+1}$.
The hypercubes must be properly stacked without overhang---recall for this notion the familiar picture of lozenge tilings, which corresponds to $d=2$.
If a hypercube is present at some coordinate $\mathbf x\in\mathbb Z^{d+1}$,
then we require the presence of a hypercube at $\mathbf x-\mathbf e_i$
for every $1\leq i\leq d+1$;
this indeed ensures that every hypercube is well-supported,
and excludes overhang.

The formal definition is as follows.
If $\x,\y\in\mathbb R^{d+1}$,
then write $\x\leq\y$ whenever $\x_i\leq\y_i$
for all $i$.
A set $A\subset\R^{d+1}$
is called \emph{closed under $\leq$} whenever
$\x\in A$ and $\y\leq\x$ implies $\y\in A$.
Write $L(A)$
for the closure of a set $A\subset\R^{d+1}$
under $\leq$,
that is,
\[
  L(A):=\{\y\in\R^{d+1}:\text{$\y\leq \x$ for some $\x\in A$}\}.
\]
A \emph{stepped surface} is a strict nonempty subset
of $\R^{d+1}$ of the form $L(A)$
for some $A\subset\Z^{d+1}$.
Let $\Psi$ denote the set of stepped surfaces.
If $F$ is a stepped surface, then write $V(F)$
for the set $\partial F\cap\mathbb{Z}^{d+1}$.
The set $V(F)$ should be thought of as the discrete boundary of $F$.
It is a simple exercise to work out that $F=L(V(F))$. In particular,
each stepped surface is characterised by this discrete boundary.

  \subsection{The height function of a stepped surface}
  \label{subsec:geo_overview:height_function}
  Consider a stepped surface $F$.
  The height function associated to $F$ is essentially a function that has the discrete boundary $V(F)$
  of $F$ as its graph.
  The value of this function at each vertex $\x\in V(F)$ is given by $\x_1+\cdots+\x_{d+1}=(\x,\n)$,
  where $\n$ is the vector $\e_1+\cdots+\e_{d+1}$
  and $(\cdot,\cdot)$ the natural inner product.
  This value is also called the \emph{height} of the vertex $\x$.

  Write $(X^d,E^d)$ for the graph obtained from the square lattice
  $\Z^{d+1}$ after identifying all vertices which differ by an integer multiple
  of $\n$.
  In particular, each vertex $[\x]\in X^d$ is an equivalence class
  of the form $[\x]:=\x+\Z\n$
  for some $\x\in\Z^{d+1}$.
  The $2d+2$ neighbours of $[\x]$ are of the form $[\x\pm\e_i]$.
  The graph $(X^d,E^d)$ is called the \emph{simplicial lattice}.
  It is a simple exercise to derive from the definition of a stepped
  surface that $|[\x]\cap V(F)|=1$ for any $[\x]\in X^d$.
  To derive this, observe first that $\x+\R\n$ intersects $\partial F$ once because $F$ is
  closed under $\leq$,
  then note that this point of intersection has integral coordinates
  whenever
  $\x$ does.

  The \emph{height function} associated to
  $F$ is given by the function
  \[
    f:X^d\to\Z,\,[\x]\mapsto(\y,\n)~\text{where}~\{\y\}=[\x]\cap V(F).
  \]
  This is the desired parametrisation of $V(F)$;
  the set $V(F)$ is equal to
  \[
    V(f):=\{\x\in\Z^{d+1}:(\x,\n)=f([\x])\}.
  \]
  The assignment $F\mapsto f$ is injective because
  $F\mapsto V(F)$ is injective.

  Let us now identify the image of the map $F\mapsto f$.
  The function $f$ satisfies
  $f([\x])\equiv (\x,\n)\mod d+1$;
  we call this the \emph{parity condition}.
  Assert furthermore that $f([\x+\e_i])\leq f([\x])+1$.
  This assertion is called the \emph{Lipschitz condition}.

  Suppose that the assertion is false, that is, that instead $f([\x+\e_i])>f([\x])+1$.
  Write $\y$ for the unique vertex in $[\x]\cap V(F)$,
  and write $\z$ for the unique vertex in $[\x+\e_i]\cap V(F)$.
  Then $\z_j> (\y+\e_i)_j\geq \y_j$ for all $j$.
  In this case, it is impossible that both $\z$ and $\y$ are
  contained in $\partial F$, because
  $F$ is closed under $\leq$.
  This proves the assertion.

  The Lipschitz condition also implies that $f([\x+\e_i])\geq f([\x])-d$,
  since $[\e_1+\cdots+\e_{d+1}]=[\n]=[\0]$.
  The Lipschitz condition, this new inequality, and the parity condition imply jointly that
  the gradient $\nabla f$ of $f$ satisfies
  \[
  \nabla f([\x],[\x+\e_i]):=f([\x+\e_i])-f([\x])\in\{-d,1\}.
  \]
  On the other hand, if the flow $\nabla f$
  satisfies this equation for any $\x $ and $i$ and if $f([\0])$
  is an integer multiple of $d+1$,
  then $f$ also clearly satisfies the parity condition.

  A \emph{height function} is a function $f:X^d\to\Z$
  which satisfies $f([\0])\in(d+1)\Z$
  and
  \begin{equation*}
    % \label{eq:height_function_derivative_condition}
    \nabla(f([\x]),f([\x+\e_i]))=f([\x+\e_i])-f([\x])\in\{1,-d\}
  \end{equation*}
  for all $\x$ and $i$;
  write $\Omega$ for the set of height functions.
  If $F$ is a stepped surface,
  then the associated height function $f$ is indeed an element of $\Omega$.
  The injective map $\Psi\to\Omega,\,F\mapsto f$ is in fact
  a bijection; its inverse is given by
  the map $\Xi:\Omega\to\Psi,\,f\mapsto L(V(f))$.
For details, refer to~\cite{Linde}.

  \subsection{The tiling associated to a stepped surface}
  \label{subsection:tiling_of_a_stepped_surface}
  If $f$ is a height function, then write
\[
T(f):=\{\{\x,\y\}\in E^d:\nabla f(\x,\y)=-d\}.\]
The set $T(f)$ characterises $\nabla f$,
and therefore it characterises the function $f$ up to constants.
The first goal is to characterise the image of the map
$f\mapsto T(f)$ over $\Omega$.

A path $(\s_k)_{0\leq k\leq n}\subset X^d$ of length $n=d+1$ is called a \emph{rooted simplicial loop}
or simply a \emph{simplicial loop}
if there exists a permutation $\xi\in S_{d+1}$ such that
$\s_k=\s_{k-1}+\e_{\xi(k)}$ for any $1\leq k\leq d+1$.
This implies that $\s$ is closed because $[\x]+\e_1+\dots+\e_{d+1}=[\x+\n]=[\x]$.
Write $R^d$ for the set of rooted simplicial loops.
Sometimes we are not concerned with the starting points of the loops.
In those cases, two loops are considered equal if they are equal up to indexation---this is equivalent to requiring that the two loops traverse the same set of edges.
Write $U^d$ for the set of \emph{unrooted simplicial loops}.

Let $\s$ denote a simplicial loop.
Knowing that $\nabla f$ must integrate to zero along this simplicial loop, it is immediate
that exactly one of the edges of $\s$ is contained in $T(f)$.
Write $\Theta$ for the set of \emph{tilings},
that is, the set of subsets $T\subset E^d$
with the property that $|T\cap\s|=1$
for any simplicial loop $\s$.
To see that the map $\Omega\to\Theta,\,f\mapsto T(f)$
is surjective, let $T\in\Theta$,
and write $\alpha_T$ for the unique flow on $(X^d,E^d)$
such that
\begin{equation}
  \label{eq:definition_of_a_flow}
\alpha_T([\x],[\x+\e_i])=
\begin{cases}
  1&\text{if $\{[\x],[\x+\e_i]\}\not\in T$,}\\
  -d&\text{if $\{[\x],[\x+\e_i]\}\in T$.}
\end{cases}
\end{equation}
Then $\alpha_T=\nabla f$ for some height function $f\in\Omega$
if and only if $\alpha_T$ is conservative.
The flow $\alpha_T$ integrates to zero along any simplicial loop,
by definition of a tiling.
It is a simple exercise in group theory to see that this implies
that $\alpha_T$ integrates to zero along any closed path,
since the graph $(X^d,E^d)$ can be written as the Cayley graph
on the generators $\e_1,\dots,\e_{d+1}$
subject to the relators which are given by all possible permutations of these $d+1$
elements---these relators correspond exactly to the simplicial loops.
This proves that $T=T(f)$ for some height function $f\in\Omega$,
and therefore the map $\Omega\to\Theta,\,f\mapsto T(f)$ is surjective.
Since $T(f)$ characterises $f$ up to constants,
this also implies that the map $\Phi:\Omega\to (d+1)\Z\times \Theta,\, f\mapsto (f(0),T(f))$
is a bijection.

\subsection{The geometrical intuition behind tilings}
Let us connect the combinatorial tilings
$T\in\Theta$ to the familiar geometric picture of lozenge tilings.
If $d=2$, then $(X^d,E^d)$ is the triangular lattice,
and the set $T\subset E^d$ has the property that $T$
contains exactly one edge of every triangle.
Thus, if we remove the set $T$ from this triangular lattice,
then we are left with a collection of lozenges.

The boundary $\partial F\subset\R^{d+1}$ of a stepped surface is a union of
hypercubes of codimension one with integral coordinates.
Write $H$ for the orthogonal complement of $\n$,
and write $P:\R^{d+1}\to H$ for orthogonal projection onto
$H$.
For each stepped surface $F$, the projection map $P$
restricts to a bijection from $\partial F$ to $H$.
If $d=2$, then $P$ maps each two-dimensional square
 contained in $\partial F$,
to a lozenge embedded in $H$,
and jointly these lozenges partition $H$---we ignore here the fact that the topological boundaries of the lozenges overlap.
We observed in the previous paragraph that each edge in $T$
encodes exactly one of these lozenges.
The same reasoning applies to higher dimensions:
the map $P$ maps the hypercubes of codimension
$1$ which make up $\partial F$ to $H$,
and these projected hypercubes partition the $d$-dimensional real vector
space $H$ up to overlapping boundaries.
Finally, each edge in $T$ encodes exactly one of these projected
hypercubes.

\subsection{Lemmas for analysing stepped surfaces}

\subsubsection{The map $\Xi$ preserves the lattice structure}
% The map $\Xi$ is an order-preserving bijection from $(\Omega,\leq)$ to $(\Psi,\subset)$.
% Moreover, taking the pointwise maximum of two height functions
% is equivalent to taking
% the union of the corresponding stepped surfaces.
% Taking the pointwise minimum corresponds to taking an intersection.
% This is summarised in the following proposition.

\begin{lemma}
  \label{lemma:aux}
  Let $f_1$ and $f_2$ be height functions and let $F_1=\Xi(f_1)$ and $F_2=\Xi(f_2)$.
  Then
  \begin{enumerate}
    \item
    \label{lemma:aux:orderwell}
    $f_1\leq f_2$ if and only if $F_1\subset F_2$,
    \item $f_1\vee   f_2$ is a height function,
    $F_1\cup F_2$ is a stepped surface, and $\Xi(f_1\vee f_2)=F_1\cup   F_2$, \label{lemma:aux_1}
    \item
    \label{lemma:aux_2}
    $f_1\wedge f_2$ is a height function, $F_1\cap F_2$ is a stepped surface, and $\Xi(f_1\wedge f_2)=F_1\cap F_2$.
  \end{enumerate}
  Of course,~\ref{lemma:aux_1} and~\ref{lemma:aux_2} extend to finite unions, intersections, maxima, and minima.
\end{lemma}

See Lemma~2 in~\cite{Linde} for a proof.

\subsubsection{Local moves}\label{subsec:localmove}
If two height functions agree
at all but one vertex, then it is said that they differ by a  \emph{local move}.
A local move is equivalent to adding or removing a single unit hypercube to the corresponding stepped surface.

\begin{lemma}
  \label{lem:localMove}
  Suppose that $R\subset X^d$ is finite
  and that $f$ and $g$ are height functions that are equal outside
  $R$ and satisfy $f\leq g$ on $R$.
  Then there exists a sequence of height functions
  $(f_k)_{0\leq k\leq n}\subset \Omega$
  with
  \begin{enumerate}
    \item $f_0=f$ and $f_n=g$,
    \item for every $0\leq k<n$, there is a unique
    $\mathbf x\in R$ such that
    $f_{k+1}=f_k+(d+1)\cdot 1_{\mathbf x}$.
  \end{enumerate}
  The sequence is increasing and all functions $f_k$ agree
  to $f$ and $g$ on $X^d\smallsetminus R$.
\end{lemma}

See Lemma~3 in~\cite{Linde} for a proof.
If $f$ and $g$ agree outside $R$ but neither $f\leq g$ nor $f\geq g$,
then one can first go down from $f$ to $f\wedge g$ and then up from $f\wedge g$
to $g$. Lemma~\ref{lemma:aux} ensures that $f\wedge g$ is a height function.

\subsubsection{The Kirszbraun theorem}
Consider a vertex $[\x]\in X^d$.
Since all elements in $[\x]$ differ from one another by multiples
of $\n$, and since $\n\in\operatorname{Ker}P$,
there is a unique element $\y\in H$
such that $\{\y\}=P[\x]$.
Let us identify each vertex $[\x]\in X^d$
with this corresponding point $\y\in H$.
Write $\g_i:=P\e_i=\e_i-\n/(d+1)$.
The neighbours of $\x\in X^d\subset H$
are then given by $\x\pm\g_i$.

Define the asymmetric norm $\|\cdot\|_+$ on $H$ by
$\|\mathbf x\|_+:=-(d+1)\min_i \mathbf x_i$.
Remark that $\|\cdot\|_+$ is the largest asymmetric norm on $H$
subject to $\|\g_i\|_+\leq 1$ for all $i$.
In other words, a function
$f:X^d\to\R$ satisfies the Lipschitz condition introduced in Subsection~\ref{subsec:geo_overview:height_function}
if and only if
\begin{equation}
  \label{eq:Lipschitz_generalized}
  f(\y)-f(\x) \leq \|\y-\x\|_+
\end{equation}
for all $\x,\y\in X^d$.
A function $f:D\to\R$ defined on a subset $D\subset H$
is called \emph{Lipschitz} whenever $f$ satisfies~\eqref{eq:Lipschitz_generalized}
for all $\x,\y\in D$.
This definition is consistent with the previous definition of the Lipschitz property
for height functions.

\begin{lemma}
  \label{lemma:HFisLIP}
   % A function $f:X^d\to\mathbb R$ is a height function if and only if
   %    $f$ is Lipschitz and $f(\mathbf x)\in\operatorname{Parity}(\mathbf x)$
   %    for every $\mathbf x\in X^d$.
    If a function $f:H\to\mathbb R$ is Lipschitz then there exists
      a unique largest height function $g$ subject to $g\leq f|_{X^d}$;
      the value of $g$ at each
       $\mathbf x\in X^d$ is given by $g(\mathbf x):=k$,
      where $k$ is the largest integer which
      makes $g$ satisfy the parity condition at $\x$,
      and which does not exceed $f(\x)$.
\end{lemma}

\label{pagerefLipfloor}
If $f:H\to\mathbb R$ is a Lipschitz function,
then write $\lfloor f\rfloor$ for the largest height function subject to $\lfloor f\rfloor\leq f|_{X^d}$;
its function values are given by the previous lemma.
We leave its proof to the reader.
The previous lemma results in a discrete analogue of the Kirszbraun theorem for the current setting.

\begin{lemma}
  \label{lem:NEW_Kirszbraun}
If $R\subset H$ and $f:R\to \mathbb R$ is Lipschitz,
then $f$ extends to a Lipschitz function $\bar f:H\to \mathbb R$.
If $R\subset X^d$ and $f:R\to \mathbb Z$ is Lipschitz and satisfies
the parity condition
for every $\mathbf x\in R$,
then $f$ extends to a height function $\bar f\in\Omega$.
\end{lemma}

The first assertion is the original Kirszbraun theorem.
For the second assertion, let $R\subset X^d$ and $f:R\to\mathbb Z$
be as in the lemma. The function $f$ extends to some Lipschitz function
$g:H\to\mathbb R$ by the Kirszbraun theorem. The previous lemma states that $\lfloor g\rfloor$ is a height function that equals
 $f$ on $R$. This proves the second assertion.

% \input{sections/20_introduction_alt.tex}
% \input{sections/30_height_functions_alt.tex}
% \input{sections/40_simplicial_lattice.tex}
%!TEX root= main_tmp.tex

\section{Random height functions}
\label{sec:random_height_functions}

In this section, we introduce boundary
conditions (fixed or periodic)
in order to reduce $\Omega$
to a finite set, and define and study probability
measures on these finite sets.
We assert that the model is monotone in boundary conditions,
and state an immediate corollary of this fact
by employing the Azuma-Hoeffding inequality.

  %!TEX root= ../main.tex
\subsection{Fixed boundary conditions}
\label{subsec:fbc}

We reduce $\Omega$ to a finite set by applying fixed
boundary conditions.
One can study the uniform probability measure on this finite set.
One can also define more general Boltzmann measures.
% Boltzmann measures are defined in Subsection~\ref{subsec:boltzmann_measures}.
% In Subsection~\ref{subsec:fbc:variancebound}, we use the cluster boundary swap from the previous section
% to derive a bound on the variance of the random variable $f(\mathbf x)$ (where $f$ is the random height function
% and $\mathbf x$ a vertex in $X^d$).
% The proof works for general Boltzmann measures.
% \subsection{Fixed boundary conditions}
% \label{subsec:fbc_one}
If $R\subset X^d$, then write $R^\complement$ for $X^d\smallsetminus R$.
Write $E^d(R)$ for the set of edges
in $E^d$ that are incident to at least one
vertex in $R$.

\begin{definition}
  \label{def:fixedbc}
  Define, for any height function $f$ and for any tiling $T$,
  \begin{alignat*}{2}
    &\Omega(R,f)&&:=\{g\in\Omega:g|_{R^\complement}=f|_{R^\complement}\},\\
    &\Theta(R,T)&&:=\{Y\in\Theta:Y\smallsetminus E^d(R)=T\smallsetminus E^d(R)\}.
  \end{alignat*}
  Call a set $R\subset X^d$ a \emph{region} if $R$ is finite and
  if $R^\complement$ is connected.
\end{definition}

% The set $\Omega(R,f)$ should be thought
% of as the set of height functions that extend
% $f|_{R^\complement}$ to $X^d$.

\begin{lemma}
  \label{lem:FBClemma}
  Let $R\subset X^d$ be a finite set, $f$
  a height function,
  and $T:=T(f)$.
  Then
  \begin{enumerate}
    \item\label{lem:fbc_finite} $\Omega(R,f)$ and $\Theta(R,T)$ are finite sets,
    \item The map $g\mapsto T(g)$ restricts to an injection from
    $\Omega(R,f)$ to $\Theta(R,T)$,
    \item\label{lem:fbc_bijection} If $R$ is a region,
    then the restricted map
  in the previous statement is a bijection.
    % \item\label{lem:fbc_minmaxchar} There exist unique $f^-,f^+\in\Omega(R,f)$
    % such that $\Omega(R,f)=\{g\in\Omega:f^-\leq g\leq f^+\}$.
  \end{enumerate}
\end{lemma}

\begin{proof}
  Without loss of generality, $R$ does not contain $\mathbf 0$.
  If $g\in\Omega(R,f)$ then we must have $g(\mathbf 0)=f(\mathbf 0)$
  and $T(g)\smallsetminus E^d(R)=T(f)\smallsetminus E^d(R)$.
  The map $g\mapsto T(g)$ restricts to an injection because
  the gradient $\nabla g$ can be reconstructed from $T(g)$,
  which is sufficient for reconstructing $g$ as the constant is determined from $g(\0)=f(\0)$.
  % and therefore Theorem~\ref{thm:bijectionOmegTheta} implies that the map
  % $g\mapsto T(g)$ restricts to an injection from
  % $\Omega(R,f)$ to $\Theta(R,T)$.
  The logarithm of $|\Theta(R,T)|$
  is bounded by $|E^d(R)|\log 2<\infty$.
  We have now proven the first two assertions of the lemma.

  Next, we prove that the same restriction map is also surjective whenever
  $R$ is a region.
  Fix $Y\in \Theta(R,T)$
  and define $g:=\Phi^{-1}(f(\mathbf 0),Y)$.
  It suffices to show that $g\in\Omega(R,f)$,
  that is, that $g(\mathbf x)=f(\mathbf x)$ for all
  $\mathbf x\in R^\complement$.
  Let $\mathbf p$ denote a path from $\mathbf 0$
  to $\mathbf  x$ through $(X^d\smallsetminus R,E^d\smallsetminus E^d(R))$;
  such a path exists by definition of a region.
  Then $g(\mathbf x)-f(\mathbf 0)$ and $ f(\mathbf x)-f(\mathbf 0)$ can both
  be calculated in terms of integrals of
    $\nabla g$ and $\nabla f$ respectively over the path $\p$.
    Since $Y\smallsetminus E^d(R)=T\smallsetminus E^d(R)$,
    these gradients are equal on the edges of $\p$,
    which proves that $g(\x)=f(\x)$.
\end{proof}

Fix a finite set $R\subset X^d$
and a height function $f\in\Omega$.
Write $f^\pm$ for the pointwise minimum and
maximum over all height functions in the finite set $\Omega(R,f)$.
These are also height function by virtue of
Lemma~\ref{lemma:aux},
and clearly $\Omega(R,f)=\{g\in\Omega:f^-\leq g\leq f^+\}$.
The same lemma implies the following result.

\begin{lemma}
  \label{lem:onemorewayforFBC}
  Fix a finite set $R\subset X^d$ and a height function $f$. Write $f^\pm$
  as in the preceding discussion and
  define $F^\pm:=\Xi(f^\pm)$.
  Then $\Xi$ restricts to a bijection from
  $\Omega(R,f)$ to
  $\{F\in\Psi:F^-\subset F\subset F^+\}$.
\end{lemma}

% \subsection{Boltzmann measures}
% \label{subsec:boltzmann_measures}
  Next, we define Boltzmann measures.
  The uniform probability measures on
  $\Omega(R,f)$ and $\Theta(R,T)$ are examples of Boltzmann measures.
The introduction of  Boltzmann measures allows us to increase the relative probability of tilings containing certain edges.

  \begin{definition}
    Let $R$ be a region, $f$ a height function, and $T:=T(f)$ a tiling.
    A positive real function $w$ on $E^d(R)$ is called a \emph{weight function}.
    Let $\mathbb P_{w}$ be the probability measure on the set $\Theta(R,T)$
    such that
    $\mathbb P_w(Y)\propto \prod\nolimits_{e\in Y\cap E^d(R)} w(e)$ for any $Y\in\Theta(R,T)$,
    that is,
    \[
    \mathbb P_w(Y):=\frac{1}{Z_w} \prod_{e\in Y\cap E^d(R)} w(e)
    \qquad  \text{where}\qquad
    Z_w:=
    \sum_{Y\in \Theta(R,T)}\prod_{e\in Y\cap E^d(R)} w(e).
    \]
    The probability measure $\mathbb P_{w}$ is called a \emph{Boltzmann measure}
    and the normalising constant $Z_{w}$ is called the \emph{partition function}.
    The measure $\mathbb P_{w}$ is also considered a probability measure on the sample space
    $\Omega(R,f)$ by defining $\mathbb P_w(g):=\mathbb P_w(T(g))$.
    Write $\mathbb P$ for $\mathbb P_w$ with $w$ identically equal to $1$,
    and write $Z$ for the corresponding partition function.
    Observe that $Z=|\Omega(R,f)|=|\Theta(R,T)|$.
    The definition of $Z_w$ makes sense also when $w$ takes complex values.
  \end{definition}

  We prove that $Z_w$ equals the
  Cayley hyperdeterminant of a suitable hypermatrix in Section~\ref{sec:Kasteleyn}.

  %!TEX root= ../main.tex

\subsection{The periodic setting}
\label{subsection:PLSD}

Write $H^*$ for the natural dual space of the
$d$-dimensional real vector space $H$.
A linear form $s\in H^*$ is called a \emph{slope}.

\begin{definition}
  Periodic boundary conditions are characterised by a pair $(L,s)$,
  where $L\subset X^d$ is a full-rank sublattice and $s\in H^*$ a slope.
  A function $f:X^d\to\mathbb R$ is called
  \emph{$(L,s)$-periodic} if, for any $\mathbf x\in X^d$ and $\mathbf y\in L$,
  \[f(\mathbf x+\mathbf y)=f(\mathbf x)+s(\mathbf y).\]
  Write
  $\Omega(L,s)$ for the set of $(L,s)$-periodic height functions that map $\mathbf 0$ to $0$.
  Call a pair $(L,s)$ \emph{valid} if $\Omega(L,s)$ is nonempty.
\end{definition}

It is not a priori clear which periodic boundary conditions $(L,s)$ are valid.
If $f$ is a height function, then write $f|_L$
and $s|_L$ for the restrictions of $f$ and $s$ to $L\subset X^d\subset H$
respectively.
First, every function $f\in \Omega(L,s)$ must satisfy
$f|_L=s|_L$, so if $s|_L$ does not extend to a height function then
$\Omega(L,s)$ is empty.
On the other hand, if $s|_L$ extends to a height function,
then the minimum amongst all possible extensions
is $(L,s)$-periodic.
Therefore $(L,s)$ is valid if and only if
$s|_L$ extends to a height function.
Lemma~\ref{lem:NEW_Kirszbraun} imposes a Lipschitz condition and
a parity condition on $s|_L$.
Clearly $s|_L$ is Lipschitz if and only if $s$ is Lipschitz.

\begin{definition}
  Write $\mathcal S$ for the set of slopes in $H^*$
  which are Lipschitz, that is,
  \[
  \mathcal S:=\{s\in H^*:\max\nolimits_is(\g_i)\leq 1\}.
  \]
\end{definition}

\begin{proposition}
  The set $\mathcal S\subset H^*$ is a closed $d$-simplex
  with extreme points $(s^i)_{1\leq i\leq d+1}$,
  where each slope $s^i$ satisfies, for any $j$,
   \[s^i(\mathbf g_j):=\begin{cases}-d&\text{if $i= j$,}\\1&\text{if $i\neq j$.}\end{cases}\]
\end{proposition}

Introduce also the parity condition for the following result.

\begin{lemma}
Suppose given periodic boundary conditions $(L,s)$.
Then $(L,s)$ is valid if and only if
$s\in\mathcal S$ and
$s([\mathbf x])\equiv (\x,\n)\mod d+1$ for every $[\mathbf x]\in L\subset X^d$.
In particular, if $L\subset (d+1)X^d$,
then $(\x,\n)\in (d+1)\Z$ for every $[\mathbf x]\in L$,
and under this extra condition, $(L,s)$ is valid if and only if
$s\in\mathcal S$ and $s(\mathbf x)\in(d+1)\mathbb Z$ for every $\mathbf x\in L$.
\end{lemma}

If $L=n(d+1)X^d$ for some $n\in\mathbb N$,
then $(L,s)$ is valid if and only if
$s\in \mathcal S$ and
\[s(n(d+1)\mathbf g_i)=n(d+1)s(\mathbf g_i)\in (d+1)\mathbb Z\]
for every $1\leq i\leq d+1$,
that is, $s(\mathbf g_i)\in \mathbb Z/n$.

\begin{definition}
  Write $L_n:=n(d+1)X^d$ for every $n\in\mathbb N$.
  Define \[\mathcal S_n:=\left\{s\in\mathcal S:\text{$s(\mathbf g_i)\in\mathbb Z/n$ for every $1\leq i\leq d+1$}\right\}.\]
  In other words, $\mathcal S_n$ is precisely the set of slopes $s$
  such that $(L_n,s)$ is valid.
\end{definition}

Remark that $\mathcal S_n$ converges to $\mathcal S$ in the Hausdorff metric,
in the sense that
every slope $s\in \mathcal S$ can be approximated by a sequence
of slopes $(s_n)_{n\in\mathbb N}$ where  $s_n\in \mathcal S_n$ for each $n$.

\subsection{Symmetries of the periodic setting}
For $\mathbf x\in X^d$, write $\theta_\mathbf x$
for the map $H\to H,\,\mathbf y\mapsto \mathbf y+\mathbf x$.
This map is called a \emph{shift},
and it is also clearly a symmetry of $X^d$.
Write $\Theta(L)$
for the group $\{\theta_\mathbf x:\mathbf x\in L\}$
for any sublattice $L$ of $X^d$,
and write $\Theta:=\Theta(X^d)$.

If $f$ is a function defined on either $H$ or $X^d$,
then write $\theta f$ for the function defined by
$\theta f(\mathbf x):=f(\theta \mathbf x)$.
% If $\alpha$ is a flow on $(X^d,E^d)$,
% then write $\theta\alpha$ for the flow defined by $\theta\alpha(\mathbf x,\mathbf y):=\alpha(\theta\mathbf x,\theta\mathbf y)$.
If $f$ is a height function and $\theta\in\Theta$,
then define the height function $\tilde\theta f$ by $\tilde \theta f:=\theta f-f(\theta\mathbf 0)+f(\mathbf 0)$.
% \begin{equation}
%   \label{eq:definitionThetaF}
%   (\tilde\theta f)(\mathbf x):=f(\theta\mathbf x)-f(\theta\mathbf 0)+f(\mathbf 0).
% \end{equation}
In other words, $\tilde\theta f$ is the unique height function such that
$(\tilde\theta f)(\mathbf 0)=f(\mathbf 0)$
and $\theta T(\theta f)=T(f)$.
The map $\tilde\theta:\Omega\to\Omega$ is bijective,
and $\tilde\theta$ restricts to a bijection from $\Omega(L,s)$
to $\Omega(L,s)$ for any periodic boundary conditions $(L,s)$.

Each height function $f\in \Omega(L,s)$
is characterised by
the $\Theta(L)$-invariant set
 $T(f)\subset E^d$.
 % This is due to Theorem~\ref{thm:bijectionOmegTheta}
 % and the fact that $f(\mathbf 0)=0$.
This implies in particular that
the set $\Omega(L,s)$ is finite, because
$|\Omega(L,s)|\leq 2^{|E^d/\Theta(L)|}<\infty$.

\begin{lemma}
  \label{lem:expectation_periodic}
Pick valid periodic boundary conditions $(L,s)$,
write $\mathbb P$ for the uniform probability measure on $\Omega(L,s)$,
and write $f$ for the random function in $\Omega(L,s)$.
Then $\tilde\theta f\sim f$ in $\mathbb P$ for any $\theta\in\Theta$.
Moreover, $\mathbb Ef(\mathbf x)=s(\mathbf x)$ for every $\mathbf x\in X^d$.
\end{lemma}

\begin{proof}
  The first assertion is obvious as $\tilde\theta:\Omega(L,s)\to\Omega(L,s)$
  is a bijection and $\mathbb P$ is uniform on this set.
  The first assertion implies that the map $\mathbb E(f(\cdot)):X^d\to\mathbb R$
  is additive.
  Therefore it must extend to a linear form in $H^*$.
  Now $L$ is full-rank and
  $\mathbb P(f(\mathbf x)=s(\mathbf x))=1$ for every $\mathbf x\in L$,
  and therefore $\mathbb E(f(\cdot))$ must extend to the linear form $s\in \mathcal S\subset H^*$.
\end{proof}

\subsection{Monotonicity of the random function}
\label{subsec:mono}

Monotonicity in boundary conditions is often an essential
property for understanding the macroscopic behaviour of the system.
See Lemmas~3 and~4 in \cite{Linde}
for a proof of the following theorem---note that these lemmas apply both to
fixed boundary conditions in the general Boltzmann setting,
as well as to periodic boundary conditions.
The proof simply says that the Glauber dynamics
preserve the monotonicity and mix to the correct distribution.

\begin{theorem}[Monotonicity]
  \label{stabmon}
  Let $R$ be a region,
   let $b_1,b_2\in\Omega$, and fix a weight function $w$.
  Write $\mathbb P_1$ and $\mathbb P_2$ for the Boltzmann measures with weight $w$
  on $\Omega(R,b_1)$ and $\Omega(R,b_2)$ respectively.
  Write $a_-$ and $a_+$ for the infimum and supremum of $(b_1-b_2)|_{R^\complement}$
  respectively.
  Then there exists a probability distribution $\mathbb P$
  on the pair $(f_1,f_2)\in \Omega(R,b_1)\times\Omega(R,b_2)$
  with marginals $\mathbb P_1$ and $\mathbb P_2$
  such that
   $a_-\leq f_1-f_2\leq a_+$
  almost surely.

  Now let $(L,s)$ denote a valid pair of periodic boundary
  conditions,
   let $b_1,b_2\in\Omega(L,s)$,
  and fix $R\subset X^d$.
  Write $\P_1$ and $\P_2$
  for the uniform probability measure on $\Omega(L,s)$
  conditioned on $f|_R=b_1|_R$
  and $f|_R=b_2|_R$ respectively.
  Write $a_-$ and $a_+$ for the infimum and supremum of $(b_1-b_2)|_R$
  respectively.
  Then there exists a probability distribution $\mathbb P$
  on the pair $(f_1,f_2)\in \Omega(L,s)^2$
  with marginals $\mathbb P_1$ and $\mathbb P_2$
  such that
   $a_-\leq f_1-f_2\leq a_+$
   almost surely.
\end{theorem}

  %!TEX root= ../main.tex

\subsection{Pointwise concentration inequalities}
\label{subsec_Pointwise_concentration_inequalities}

\begin{theorem}
  \label{theorem:Azuma}
  Let $R\subset X^d$ denote a finite set,
   let $b\in\Omega$, and fix a weight function $w$.
  Write $\mathbb P$ for the Boltzmann measure with weight $w$
  on $\Omega(R,b)$.
  Fix $\x\in X^d$.
  Then the following inequalities hold true:
  \begin{enumerate}
    \item
    \label{AH_ineq_VAR}
     $\operatorname{Var} f(\x)\leq (d+1)^2 n$,
    \item
    \label{AH_ineq_UPPER}
    $\P(f(\x)-\mu\geq(d+1)a)\leq \exp - \frac {a^2}{2n}$ for all $a\geq 0$
    whenever $n>0$,
    \item
    \label{AH_ineq_LOWER}
    $\P(f(\x)-\mu\leq(d+1)a)\leq \exp - \frac {a^2}{2n}$ for all $a\leq 0$
    whenever $n>0$,
  \end{enumerate}
  where $n=d_{(X^d,E^d)}(\x,R^\complement)$
  and $\mu=\E f(\x)$.

  Now let $(L,s)$ denote a valid pair of periodic boundary
  conditions, and write $\P$
  for the uniform probability measure on $\Omega(L,s)$.
  Then~\eqref{AH_ineq_VAR}--\eqref{AH_ineq_LOWER}
  remain true for the choices $n=d_{(X^d,E^d)}(\x,L)\leq d_{(X^d,E^d)}(\x,\0)$
  and $\mu=\E f(\x)=s(\x)$.
\end{theorem}

\begin{proof}
This is a simple consequence of
monotonicity
(Theorem~\ref{stabmon})
and the Azuma-Hoeffding inequality.
Focus on fixed boundary conditions;
the proof for periodic boundary conditions is entirely analogous.
Let $\p=(\p_k)_{0\leq k\leq n}$
denote a path of shortest length from
$R^\complement$ to $\x$.
Write $(f_k)_{0\leq k\leq n}$
for the martingale such that
$f_k$ equals the expectation
of $f(\x)$ conditional on the values of $f$ on the vertices
in $\{\p_0,\dots,\p_k\}$.
By Theorem~\ref{stabmon},
we have $|f_k-f_{k+1}|\leq d+1$.
The theorem now follows from the Azuma-Hoeffding inequality.
\end{proof}

%!TEX root= ../main.tex
\section{The cluster boundary swap}
\label{sec:CBS}

In the seminal work~\cite{SHEFFIELD},
Sheffield introduces \emph{cluster swapping}.
This technique is related to the double
dimer model,
where for uniform probability measures,
the orientation of each loop is uniformly random
in the two states, independently of any other structure that is present.
In this section, we introduce the \emph{cluster boundary swap}.
The cluster boundary swap is a special case of the cluster swap,
adapted and optimised for
the special nature of the model under consideration.
Its properties are reminiscent of the double dimer model,
because, conditional on the geometrical structure involving boundaries
and double edges, the orientation of each boundary
is uniformly random in its two states.

\subsection{The boundary graph and the level set decomposition}

In Subsection~\ref{subsection:tiling_of_a_stepped_surface}
it was proven
 that every height function $f$ is characterised by the pair
$\Phi(f)=(f(\mathbf 0),T(f))$.
In this section, $f_1$ and $f_2$ denote height functions,
and we write $(a_i,T_i):=\Phi(f_i)$
for $i\in\{1,2\}$.
The difference function $f_1-f_2$ is denoted by $g$.
The goal of this subsection is to understand the level set structure of $g$.

\begin{lemma}\label{important}
Let $\mathbf s\in R^d$ be a rooted simplicial loop. As one walks along $\mathbf s$,
\begin{enumerate}
  \item The function $f_1$ moves up by $1$ exactly $d$ times,
  \item \label{important_two}The function $f_1$ moves down by $d$ exactly once,
  \item \label{important_three}Either $g$ remains constant, or it changes value twice,
  \item \label{important_four}If $g$ is not constant, then the difference between its two values is $d+1$.
\end{enumerate}
\end{lemma}

The lemma follows immediately from the observations in Subsection~\ref{subsection:tiling_of_a_stepped_surface}.

Write $A\ominus B$ for the symmetric difference of arbitrary sets $A$ and $B$.

% !TEX root =  ../main.tex

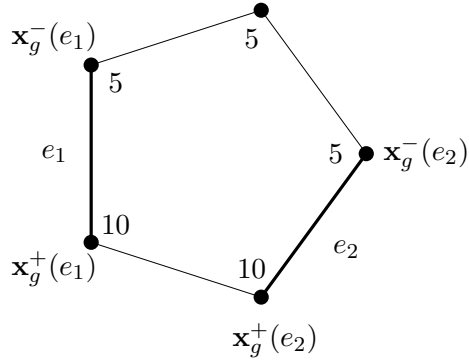
\begin{figure}
\begin{center}
  \begin{tikzpicture}[x=1cm,y=1cm,scale=1]
   \draw[ very thin] (0:2cm) -- (1*72:2cm) -- (2*72:2cm) -- (3*72:2cm) -- (4*72:2cm) -- cycle;
   \draw[very thick] (2*72:2cm) -- (3*72:2cm);
   \draw[very thick] (0*72:2cm) -- (4*72:2cm);

   \newcommand{\pentalable}[3]{
    \fill (#1*72:2cm) circle (0.1);
    \node at (#1*72:1.6cm) {$#2$};
    \node at (#1*72:2.6cm) {$#3$};
   }
   \pentalable{0}{5}{\quad\mathbf x_g^-(e_2)}
   \pentalable{1}{5}{}
   \pentalable{2}{5}{\mathbf x_g^-(e_1)}
   \pentalable{3}{10}{\mathbf x_g^+(e_1)}
   \pentalable{4}{10}{\mathbf x_g^+(e_2)}

   \node at (2.5*72:2.1cm) {$e_1$};
   \node at (4.5*72:2.15cm) {$e_2$};

\end{tikzpicture}

\caption[A simplicial loop]{
The values of $g=f_1-f_2$ along a simplicial loop; $d=4$
}

\label{fig:sl}

\end{center}
\end{figure}

\begin{definition}
  Define the graph $G_g=(V_g,E_g)$ as follows.
  Its vertex set $V_g$ is given by
  \[V_g:=T_1\ominus T_2=\{e\in E^d:\text{$g$ is not constant on $e$}\}\subset E^d,\]
  and two vertices $e_1,e_2\in V_g\subset E^d$
  are neighbours if some simplicial loop travels through both $e_1$
  and $e_2$.
  The graph $G_g$ is called the \emph{boundary graph}.
  For $e\in V_g$, write $\mathbf x_g^-(e),\mathbf x_g^+(e)\in X^d$ for the vertices contained in $e\subset X^d$ on which
  $g$ takes the smaller value and the larger value respectively---see Figure~\ref{fig:sl}.
  For any $C\subset V_g$, we write $\mathbf x_g^\pm(C):=\{\mathbf x_g^\pm(e):e\in C\}$.
\end{definition}

For example, if $d=2$, then one can identify each edge of the triangular lattice
with the lozenge obtained by removing that edge.
The set $V_g$ is then precisely the set of edges of lozenges which appear in exactly one
of the two configurations.
Two edges in $V_g$ are neighbours if they belong to the same triangle of the triangular lattice.
Each connected component of $G_g$ corresponds to the set of edges of the triangular lattice
crossed by a nontrivial loop of the double dimer model.

\begin{lemma}\label{twocomponents}
  Let $C\subset V_{g}$ be a connected component of $G_g$.
  Then $(X^d,E^d\smallsetminus C)$ consists of two connected components,
  one containing $\mathbf x_g^-(C)$, and the other containing $\mathbf x_g^+(C)$.
  Moreover, each of $\mathbf x_g^-(C)$ and $\mathbf x_g^+(C)$ is contained
  in a connected component of the graph $(X^d,E^d\smallsetminus V_{ g})$.
\end{lemma}

\begin{proof}
  Suppose that the $G_g$-vertices $e_1$ and $e_2$ are neighbours in the graph $G_g$;
  write $\mathbf s$ for a simplicial loop
  passing through both $e_1$ and $e_2$.
  Then $\mathbf s$ contains no other edges in $V_g$ by Proposition~\ref{important},~\ref{important_three},
  and therefore
   $\mathbf x_g^-(e_1)$ and $\mathbf x_g^-(e_2)$
   are connected
  in the graph $(X^d,E^d\smallsetminus V_g)$;
  see also Figure~\ref{fig:sl}.
  Induct on this argument to see that
   $\mathbf x_g^-(C)$ is contained in a connected component of $(X^d,E^d\smallsetminus V_{ g})$.
  Identical reasoning applies to the set $\mathbf x_g^+(C)$,
  and we also learn that each of $\mathbf x_g^\pm(C)$ is contained in a connected component
  of $(X^d,E^d\smallsetminus C)$.

  The sets $\mathbf x_g^\pm(C)$ cover all the endpoints of edges in $C$,
  and therefore two possibilities remain:
  either the graph $(X^d,E^d\smallsetminus C)$ is connected,
  or it consists of two connected components,
  with one containing $\mathbf x_g^-(C)$, and the other containing $\mathbf x_g^+(C)$.
  To establish the lemma we must exclude the first possibility.
  Every simplicial loop intersects $C$ an even number of times.
  The group theory arguments that proved
  that the flow in Subsection~\ref{subsection:tiling_of_a_stepped_surface}
  was conservative, imply here
  that any closed walk
  through $(X^d,E^d)$ intersects
  $C$ an even number of times.
  This proves that $(X^d,E^d\smallsetminus C)$ is not connected.
\end{proof}

\begin{definition}
  A \emph{$g$-level set} %or simply a \emph{level set}
  is a connected component of the graph $(X^d,E^d\smallsetminus V_g)$.
  A \emph{$g$-boundary} %or simply a \emph{boundary}
  is a connected component of the graph $G_g=(V_g,E_g)$.
  The $g$-level sets are considered subsets of $X^d$,
  and the $g$-boundaries are considered subsets of $V_g\subset E^d$.
  If $E$ is a $g$-boundary, then write
  $X_g^\pm(E)$ for the $g$-level set containing $\mathbf x_g^\pm(E)$.
  The \emph{level set decomposition of $g$} or $\operatorname{LSD}(g)$ is an undirected graph,
    where the vertices are the $g$-level sets
   and the edges are the $g$-boundaries.
   The $g$-boundary $E$ connects the $g$-level sets $X_g^-(E)$ and
   $X_g^+(E)$.
   Write $g$ for the graph homomorphism $g:\operatorname{LSD}(g)\to (d+1)\mathbb{Z}$
   that assigns the value $g(X)$ to a $g$-level set $X$.
   The vector field $\nabla g$ directs the edges in $\operatorname{LSD}(g)$:
   it orients each $g$-boundary $E$ from $X_g^-(E)$ to
   $X_g^+(E)$.
   Write $(\operatorname{LSD}(g),\nabla g)$ for this directed graph.
\end{definition}

If $d=2$, then the $g$-boundaries correspond exactly to the loops of the double dimer model.
The $g$-level sets correspond to the connected components of $\mathbb R^2$
with all loops of the double dimer model removed.
 In Figure~\ref{fig:bs} we see an example of this graph.
 % The different shades refer to the different values of $g$.
 Each $g$-level set
 contracts into a single $\operatorname{LSD}(g)$-vertex.
 The $\operatorname{LSD}(g)$-edges are comprised of the $g$-boundaries separating
 the $g$-level sets.

\begin{lemma}
  \label{lem:subtleBigTreeLemma}
  The graph $\operatorname{LSD}(g)$ is well-defined and a tree.
\end{lemma}

\begin{proof}
  It follows from Lemma~\ref{twocomponents} that every
  $\nabla g$-directed
  $\operatorname{LSD}(g)$-edge has a well-defined
  startpoint and endpoint, and that removing an edge disconnects the graph.
\end{proof}

\subsection{The cluster boundary swap}

% !TEX root =  ../main.tex

\begin{figure}
\begin{center}
\begin{tikzpicture}[x=1cm,y=1cm,scale=1.2]
  \begin{scope}[shift={(0,0)}]
    \node[anchor=north] at (2.5,0) {$g$ and $(\operatorname{LSD}(g),\nabla g)$};
  \begin{scope}
    \clip (0,0) rectangle (5,5);
    \fill[Aquamarine!40] (0,0) rectangle (5,5);
    \fill[Aquamarine!70] plot [smooth cycle] coordinates {(0.5,0.5) (0.5,2)  (1.5,2) (1.5,0.5)};
    \draw[very thick,fill=Aquamarine!70] plot [smooth cycle] coordinates {(2,4.5)(4.5,4.5)(4.5,2)(2,2.5)};
    \fill[Aquamarine] plot [smooth cycle] coordinates {(3,2.5)(4.3,2.5)(4.3,4.2)(3.5,4.2)};

    \node[anchor=north] at (4.5,2) {$M$};

    \fill (1,3) circle (0.05);
    \node[anchor=south] at (1,3) {$0$};
    \fill (2.4,2.9) circle (0.05);
    \node[anchor=south] at (2.4,2.9) {$3$};
    \fill (3.8,3) circle (0.05);
    \node[anchor=south] at (3.8,3) {$6$};
    \fill (1,1.8) circle (0.05);
    \node[anchor=north] at (1,1.8) {$3$};
    \begin{scope}[very thick]
      \draw[->-=.4] (1,3) to [bend right]  (2.4,2.9) ;
      \draw[->-=.6] (2.4,2.9) to [bend right] (3.8,3);
      \draw[->-=.5] (1,3) to (1,1.8);
      \draw[dotted] (1,3) to [bend left] (-1,3);
    \end{scope}
  \end{scope}
  \end{scope}

  \begin{scope}[shift={(5.5,0)}]
    \node[anchor=north] at (2.5,0) {$g'$ and $(\operatorname{LSD}(g'),\nabla g')$};
  \begin{scope}
    \clip (0,0) rectangle (5,5);
    \fill[Aquamarine!40] (0,0) rectangle (5,5);
    \fill[Aquamarine!70] plot [smooth cycle] coordinates {(0.5,0.5) (0.5,2)  (1.5,2) (1.5,0.5)};
    \draw[very thick,fill=Aquamarine!10] plot [smooth cycle] coordinates {(2,4.5)(4.5,4.5)(4.5,2)(2,2.5)};
    \fill[Aquamarine!40] plot [smooth cycle] coordinates {(3,2.5)(4.3,2.5)(4.3,4.2)(3.5,4.2)};

    \node[anchor=north] at (4.5,2) {$M$};

    \fill (1,3) circle (0.05);
    \node[anchor=south] at (1,3) {$0$};
    \fill (2.4,2.9) circle (0.05);
    \node[anchor=south] at (2.4,2.9) {$-3$};
    \fill (3.8,3) circle (0.05);
    \node[anchor=south] at (3.8,3) {$0$};
    \fill (1,1.8) circle (0.05);
    \node[anchor=north] at (1,1.8) {$3$};
    \begin{scope}[very thick]
      \draw[->-=.6] (2.4,2.9) to [bend left] (1,3);
      \draw[->-=.6] (2.4,2.9) to [bend right] (3.8,3);
      \draw[->-=.5] (1,3) to (1,1.8);
      \draw[dotted] (1,3) to [bend left] (-1,3);
    \end{scope}
  \end{scope}
  \end{scope}
\end{tikzpicture}

\caption[The level set decomposition and a cluster boundary swap]{
  The level set decomposition and a cluster boundary swap by $M$;  $d=2$.
}

\label{fig:bs}

\end{center}
\end{figure}
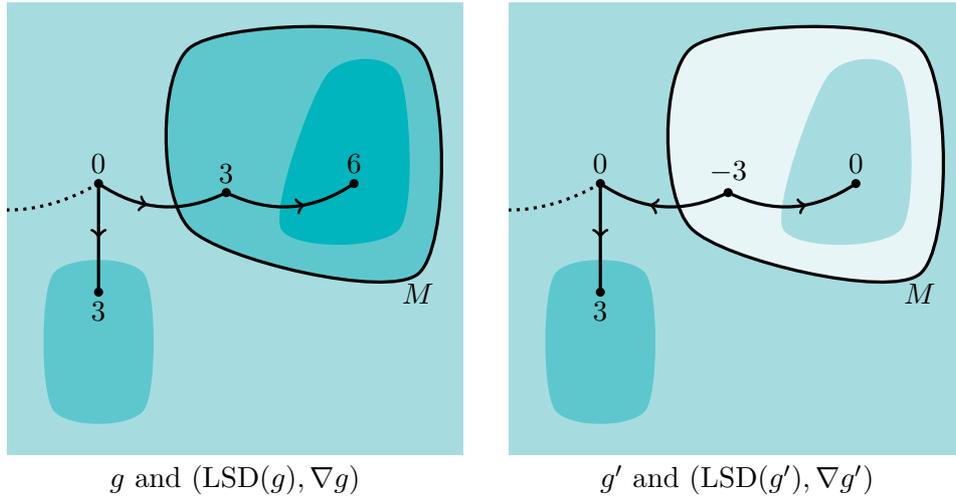

\begin{lemma}
  \label{lem:construction_of_the_CBS}
  Let $M\subset V_g\subset E^d$ be a union of $g$-boundaries.
  Then the sets $T_i':=T_i\ominus M$ are tilings for $i\in\{1,2\}$.
  Write $f'_1$ and $f_2'$ for the unique height functions
  such that $\Phi(f'_i)=(a_i,T_i')$ and define $g'=f_1'-f_2'$.
  Then $G_{g'}=G_g$ and $\operatorname{LSD}(g')=\operatorname{LSD}(g)$.
  Moreover, $\nabla g'$ and $\nabla g$ are the same except that the $g$-boundaries contained in $M$ have reversed orientation,
  that is, $\nabla g'=(-1)^{1_M}\cdot\nabla g$.
\end{lemma}

\begin{proof}
 First claim that $T_1'$ and $T_2'$ are tilings.
 We focus on $T_1'$.
 Let $\mathbf s$ be a simplicial loop,
 and abuse notation by writing $\mathbf s$ also for the set of edges
  crossed by this loop.
  It suffices to prove that $|T_1'\cap\mathbf s|=1$.
 Now either $\mathbf s\cap M$ is empty, or contains
 two edges, one from $T_1$ and one from $T_2$.
 In the former case we have
 $T_1'\cap\mathbf s=T_1\cap\mathbf s$ and consequently $|T_1'\cap\mathbf s|=1$.
 In the latter case, we have
 $T_1'\cap \mathbf s=T_2\cap\mathbf s$ and consequently $|T_1'\cap \mathbf s|=1$, as desired.
 This proves the claim.
 The appropriate functions $f_1'$ and $f_2'$ exist because
 $\Phi$ is a bijection from $\Omega$
 to $(d+1)\Z\times\Theta$.
 Next, \[V_{g'}:=T_1'\ominus T_2'=(T_1\ominus M)\ominus (T_2\ominus M)=T_1\ominus T_2 = V_g,\]
   and consequently $G_{g'}=G_g$ and
    $\operatorname{LSD}(g')=\operatorname{LSD}(g)$.
Recall the definition of $\alpha_T$
in terms of $T$ in~\eqref{eq:definition_of_a_flow}.
We have
\[\nabla g'=\alpha_{T_1'}-\alpha_{T_2'}=\alpha_{T_1\ominus M}-\alpha_{T_2\ominus M}=(-1)^{1_M}\cdot (\alpha_{T_1}-\alpha_{T_2})=(-1)^{1_M}\cdot\nabla g;\]
this follows directly from
the fact that $M\subset T_1\ominus T_2$
and from the definition of $\alpha _T$.
\end{proof}

\begin{definition}
  \label{def:CBS_Official_Def}
  Define
  \begin{alignat*}{3}
    &(T_1,T_2)\ominus M&&:=(T_1',T_2')&&=(T_1\ominus M,T_2\ominus M),\\
    &(f_1,f_2)\ominus M&&:=(f_1',f_2')&&=\big(\Phi^{-1}(a_1,T_1\ominus M),\Phi^{-1}(a_2,T_2\ominus M)\big),
  \end{alignat*}
  whenever these are related as in the previous lemma.
  Write $(f_1,f_2)\sim (f_1',f_2')$ whenever $(f_1',f_2')=(f_1,f_2)\ominus M$
  for some union of $g$-level sets $M$,
  in which case we say that the two pairs \emph{differ by a cluster boundary swap}.
  The relation $\sim$ is an equivalence relation;
  write $[(f_1,f_2)]$ for the equivalence class of $(f_1,f_2)$.
\end{definition}

\begin{remark}
 \begin{enumerate}
   \item If $(f'_1,f_2')=(f_1,f_2)\ominus M$, then $f_1'+f_2'=f_1+f_2$;
   a cluster boundary swap does not change the sum of the two involved height functions.
   To see that this is the case, observe that $M$ is a subset of $T_1\ominus T_2$,
   and therefore $1_{T_1\ominus M}+1_{T_2\ominus M}=1_{T_1}+1_{T_2}$
   and
   \[\nabla f_1+\nabla f_2=\alpha_{T_1}+\alpha_{T_2}=\alpha_{T_1\ominus M}+\alpha_{T_2\ominus M}=\nabla f_1'+\nabla f_2'.\]
  \item
    The cluster boundary swap was formalised
    in terms of the height functions
    $f_1$ and $f_2$.
    The cluster boundary swap should however
    be understood as an operation on the gradients
    $\nabla f_1$ and $\nabla f_2$
    of these height functions.
    This gradient operation is made into an operation
    on the non-gradient height functions by
    choosing the vertex $\0\in X^d$
    as a reference vertex at which
    the height is held constant.
   %
   % The notation indicates that the graphs $G_g$ and $\operatorname{LSD}(g)$
   % are constructions in terms of $g$.
   % Both graphs however are also entirely characterised by the set $V_g:=T_1\ominus T_2$.
   % The reason that we chose to have the notation refer to the function $g$ is
   % that the graph $\operatorname{LSD}(g)$ really is the level set decomposition of $g$,
   % and that the gradient $\nabla g$ directs the edges of this graph.
   % The sets $T_1$ and $T_2$ are of course linked directly to the gradients
   % $\nabla f_1$ and $\nabla f_2$.
   % This reminds us of the fact that all constructions in this section
   % are in nature constructions on the gradients of $f_1$ and $f_2$.
   % The cluster boundary swap in particular is made into an operation on $\Omega^2$
   % by choosing $\mathbf 0$ as a reference point---implicitly through using the map $\Phi$---and making sure that
   % $f_1'(\mathbf 0)=f_1(\mathbf 0)$
   % and $f_2'(\mathbf 0)=f_2(\mathbf 0)$
   % whenever
   % $(f_1',f_2'):=(f_1,f_2)\ominus M$.
   % We want to stress that
   % the graph $G_g$ and the directed graph $(\operatorname{LSD}(g),\nabla g)$
   % are invariant under adding constants to $f_1$ and $f_2$.
 \end{enumerate}
\end{remark}

Figure~\ref{fig:bs} illustrates a cluster boundary swap.
The thick contour is the set $M$,
and the two difference functions $g$ and $g'$
are related by $g=f_1-f_2$ and $g'=f_1'-f_2'$
where $(f_1',f_2'):=(f_1,f_2)\ominus M$.
Swapping by $M$ effectively inverts the orientation of the corresponding $g$-boundary.
One can swap any union of $g$-boundaries.
Therefore one can direct the edges of $\operatorname{LSD}(g)$ in any desired way.
We obtain the following theorem.

\begin{theorem}
  \label{thm:lekkerpropo}
  The relation $\sim$ is an equivalence relation on $\Omega^2$.
  The elements in the equivalence class of $(f_1,f_2)$
  correspond naturally to the graph homomorphisms from the tree $\operatorname{LSD}(g)$
  to $(d+1)\mathbb{Z}$ that map $\mathbf 0$ to $g(\mathbf 0)$.
\end{theorem}

%!TEX root= ../main.tex
\section{The variance and covariance structure}
\label{sec:varcovar}

This section is dedicated to a straightforward application of Theorem~\ref{thm:lekkerpropo} in the fixed boundary setting.

\begin{theorem}
  \label{thm:varianceboundofficial}
  Let $R$ be a region not containing $\mathbf 0$, let $b$ be a height function, and let $w:E^d(R)\to(0,\infty)$ be a weight function.
  Denote the Boltzmann measure on $\Omega(R,b)$ with weight $w$ by $\mathbb P_w$.
  Abuse notation by writing $\mathbb P_w$ for $\mathbb P_w\times\mathbb P_w$;
  write $(f_1,f_2)$ for the pair of random functions in this measure,
  and write $g:=f_1-f_2$.
  Also write $f$ for $f_1$.
  Then for any $\mathbf x\in X^d$, we have
  \[\operatorname{Var}_w f(\mathbf x)=\frac12(d+1)^2\mathbb E_w d_{\operatorname{LSD}(g)}(\mathbf 0,\mathbf x)
  %\leq \frac12(d+1)^2d_{(X^d,E^d)}(X^d\smallsetminus R,\mathbf x)
  .\]
  In other words, the variance of $f(\mathbf x)$ in $\mathbb P_w$ equals
  $\frac12(d+1)^2$ times the $\mathbb P_w$-expectation of the number of
  $g$-boundaries that separate $\mathbf x$ from $\mathbf 0$.
\end{theorem}

% Observe that the deterministic $d_{(X^d,E^d)}(X^d\smallsetminus R,\mathbf x)$
% bounds
% the random variable $d_{\operatorname{LSD}(g)}(\mathbf 0,\mathbf x)$, and therefore $\frac12(d+1)^2d_{(X^d,E^d)}(X^d\smallsetminus R,\mathbf x)$
% bounds $\operatorname{Var}_w(f(\mathbf x))$.

\begin{proof}%[Proof of Theorem~\ref{thm:varianceboundofficial}]
  The random variables $f_1(\mathbf x)$ and $f_2(\mathbf x)$ are i.i.d.,
  and therefore
  \[\operatorname{Var}_w f(\mathbf x)=\frac12\operatorname{Var}_w(f_1(\mathbf x)-f_2(\mathbf x))=\frac12\mathbb E_w(f_1(\mathbf x)-f_2(\mathbf x))^2=\frac12\mathbb E_w g(\mathbf x)^2.\]
  It suffices to prove that
  $\mathbb E_w g(\mathbf x)^2 =(d+1)^2\mathbb E_w d_{\operatorname{LSD}(g)}(\mathbf 0,\mathbf x)$.
  In fact, we make the stronger claim that
  \[\mathbb E_w(g(\mathbf x)^2|[(f_1,f_2)])=(d+1)^2d_{\operatorname{LSD}(g)}(\mathbf 0,\mathbf x).\]
  The left hand side is $\sigma([(f_1,f_2)])$-measurable by definition.
  For the right hand side,
  observe that the graph $\operatorname{LSD}(g)$
  is constant on each equivalence class $[(f_1,f_2)]$,
  which means that $d_{\operatorname{LSD}(g)}(\mathbf 0,\mathbf x)$
  is also $\sigma([(f_1,f_2)])$-measurable.
%  The notation makes sense because $\operatorname{LSD}(g)$ is constant on
%  each equivalence class $[(f_1,f_2)]$;
%  both sides are $\sigma([(f_1,f_2)])$-measurable.
  The proof of the claim relies on Theorem~\ref{thm:lekkerpropo}.

  Assert first that
  $[(f_1,f_2)]\subset\Omega(R,b)^2$
  whenever $(f_1,f_2)\in\Omega(R,b)^2$.
  The set $R^\complement$ is connected by the definition of a region,
  and it contains $\mathbf 0$.
  Therefore $R^\complement$ is contained in the $g$-level set
  containing $\mathbf 0$.
  A cluster boundary swap does not alter the values of $f_1$ and $f_2$
  on this $g$-level set,
  and therefore $f_1$, $f_2$, $f_1'$, $f_2'$, and $b$
  all assume the same values on $R^\complement$
  provided that $(f_1',f_2')\sim (f_1,f_2)$ and $(f_1,f_2)\in\Omega(R,b)^2$.
  This proves the assertion.
  Next, assert that $\mathbb P_w$ conditioned on $[(f_1,f_2)]$
  is uniform on $[(f_1,f_2)]$.
  To see that this is the case, observe that
  \[\mathbb P_w((f_1,f_2))\propto \prod_{e\in E^d(R)} w(e)^{(1_{T(f_1)}+1_{T(f_2)})(e)}.\]
  Now $1_{T(f_1)}+1_{T(f_2)}=1_{T(f_1')}+1_{T(f_2')}$
  whenever $(f_1',f_2')\sim(f_1,f_2)$, which proves the assertion.

  Theorem~\ref{thm:lekkerpropo} now provides the distribution
  of the function $g$ in the measure $\mathbb P_w$ conditioned
  on $[(f_1,f_2)]$.
  In particular,
  as $\operatorname{LSD}(g)$ is a tree,
  the distribution of $g(\mathbf x)$
  is given by summing the outcomes of
  $d_{\operatorname{LSD}(g)}(\mathbf 0,\mathbf x)$
  fair coin flips, each worth $\pm (d+1)$.
  It is well-known that the expectation of the square
  of this random variable is
  $(d+1)^2d_{\operatorname{LSD}(g)}(\mathbf 0,\mathbf x)$,
  which proves the claim.
\end{proof}

In fact, the exact same calculation works for the covariance of $f(\mathbf x)$ with $f(\mathbf y)$.

\begin{theorem}
  \label{thm:covarianceboundofficial}
  Work in the setting of the previous theorem.
  Then for any $\mathbf x,\mathbf y\in X^d$, we have
  \[\operatorname{Cov}_w(f(\mathbf x),f(\mathbf y))=\frac12(d+1)^2\mathbb E_w d_{\operatorname{LSD}(g)}(\mathbf 0,\mathbf z)
  \]
  where $\mathbf z$ is the last $\operatorname{LSD}(g)$-vertex of the $\operatorname{LSD}(g)$-path from $\mathbf 0$ to $\mathbf x$
  that also appears in the $\operatorname{LSD}(g)$-path from $\mathbf 0$ to $\mathbf y$.
  In other words, the covariance of $(f(\mathbf x),f(\mathbf y))$ in $\mathbb P_w$ equals
  $\frac12(d+1)^2$ times the expectation of the number of
  $g$-boundaries that separate both $\mathbf x$ and $\mathbf y$ from $\mathbf 0$.
\end{theorem}

\begin{proof}
  Again, we have
  $\operatorname{Cov}_w(f(\mathbf x),f(\mathbf y))=\frac12\operatorname{Cov}_w(g(\mathbf x),g(\mathbf y))$,
  % It suffices to prove that
  % \[\operatorname{Cov}_w(g(\mathbf x),g(\mathbf y))=(d+1)^2\mathbb E_w(d_{\operatorname{LSD}(g)}(\mathbf 0,\mathbf z)).\]
  and we prove that
  \[\mathbb E_w (g(\mathbf x)g(\mathbf y)|[(f_1,f_2)])=(d+1)^2d_{\operatorname{LSD}(g)}(\mathbf 0,\mathbf z).\]
  % Note that we have suggesively omitted the subscript $w$ here as $\mathbb E(\cdot|[(f_1,f_2)])$
  % is uniform on $[(f_1,f_2)]$, regardless of the actual weight.
  The conditioned measure $\mathbb P_w$ directs the edges of $\operatorname{LSD}(g)$
  independently and uniformly at random,
  as in the previous theorem.
  Thus, under the conditioned measure $\mathbb P_w$, we
  have
  \[(g(\mathbf x),g(\mathbf y))\sim (A+X,A+Y),\]
  where $A$, $X$, and $Y$ are independent, where $A$ is determined by summing the outcome of $d_{\operatorname{LSD}(g)}(\mathbf 0,\mathbf z)$
  fair independent coin flips each valued $\pm(d+1)$,
  where $X$ is determined by flipping $d_{\operatorname{LSD}(g)}(\mathbf z,\mathbf x)$ coins,
  and where $Y$ is determined by flipping $d_{\operatorname{LSD}(g)}(\mathbf z,\mathbf y)$ coins.
  This proves the assertion.
\end{proof}

%!TEX root= ../main.tex
\section{Generalisation of the Kasteleyn theory}
\label{sec:Kasteleyn}
%\subsection{Height functions as perfect matchings of the dual}
%%\subsection{Boltzmann measures}
%\subsection{The Cayley hyperdeterminant}
%\subsection{The Fourier transform in the periodic setting}

Consider fixed boundary conditions $(R,f)$ and $(R,T)$ with $R$ a region and $T=T(f)$.
The goal of this section is to show that $Z=|\Omega(R,f)|=|\Theta(R,T)|$
equals the Cayley hyperdeterminant of the adjacency hypermatrix of a suitably defined hypergraph.
In fact, we have no trouble in generalising to Boltzmann measures;
we show that one can insert the weights $w$ into the adjacency hypermatrix
so that the Cayley hyperdeterminant equals $Z_w$.
The hypergraph, which we shall denote by $(U^d,H^d)$,
is dual to the simplicial lattice $(X^d,E^d)$.
Recall that $U^d$ is the set of unrooted simplicial loops that was introduced earlier.
In dimension $d=2$ we recover exactly the theory of the dimer model on the hexagonal lattice.
% In this section we show that the number of tilings (conditional on some fixed boundary conditions) equals
% the Cayley hyperdeterminant of the adjacency hypermatrix of a suitably defined hypergraph
% that we shall denote by
% $(U^d,H^d)$. The number of tilings is the partition function of the uniform probability measure
% of the height functions model with corresponding fixed boundary conditions.
% This method is easily generalised to obtain the partition function of general Boltzmann measures.
(See Figure~\ref{fig:2d}\ref{fig:2d:hex}.)

\subsection{The dual of the simplicial lattice}
\label{subsec:convention}
In this subsection we define the hypergraph $(U^d,H^d)$.
Consider first the collection of simplicial loops.
If $\mathbf s=(\mathbf s_k)_{0\leq k\leq d+1}\in R^d$ is a rooted simplicial loop then $\mathbf s$ is characterised
by its starting point $\mathbf s_0\in X^d$
and the permutation $\xi\in S_{d+1}$
which describes in which order the increments
$(\g_i)_i$ appear.
This automatically gives a bijection from
$R^d$ to $X^d\times S_{d+1}$.
Let us agree to index each unrooted loop $\mathbf s\in U^d$ (by default)
such that $\mathbf s_1=\mathbf s_0+\mathbf g_{d+1}$.
There is a unique way of doing so, because
the increment $\mathbf g_{d+1}$ appears exactly once in each loop.
With this convention,
each unrooted loop $\mathbf s\in U^d$ is characterised
by its starting point $\mathbf s_0$ and the order $\xi\in S_d$ in which the remaining increments
$\{\mathbf g_1,\dots,\mathbf g_d\}$ appear in the path after the first increment.
By adopting the convention we obtain a bijection from $U^d$ to $X^d\times S_d$.
We identify the unrooted loop $\mathbf s$ with its image under the bijection,
so that every pair $(\mathbf x,\xi)\in X^d\times S_d$ denotes also an unrooted simplicial loop.

\begin{definition}
  For any $e\in E^d$, write $h(e)$ for the set of unrooted simplicial loops that traverse $e$.
\end{definition}

Write $e=\{\mathbf x,\mathbf x+\mathbf g_j\}\in E^d$ and
let us make a number of observations about the set $h(e)$.
First, the assignment $e\mapsto h(e)$ is injective, because
the edge $e$ is the only edge that is traversed by all loops in $h(e)$.
Secondly, there are precisely $d!$ unrooted simplicial loops that traverse $e$, since
they correspond to the $d!$ ways that we can order
the $d$ increments $(\mathbf g_i)_{i\neq j}$
that we need to walk back to $\mathbf x$ from $\mathbf x+\mathbf g_j$.
Therefore $h(e)$ contains $d!$ unrooted loops.
Finally, if $\mathbf s^1,\mathbf s^2\in h(e)$
are distinct loops identified with the pairs
$(\mathbf x^1,\xi^1),(\mathbf x^2,\xi^2)\in X^d\times S_d$,
then the permutations $\xi^1,\xi^2\in S_d$ must be distinct.
Conclude that for every $\xi\in S_d$, there is a unique
$\mathbf x\in X^d$ such that $(\mathbf x,\xi)\in h(e)$.

The reason that we introduced the map $h$ is the following.
A set $T\subset E^d$ is a tiling if and only if $h(T)$
is a partition of $U^d$, the set of simplicial loops.
Once could rephrase this statement by saying that $h(T)$ is a perfect matching of the hypergraph
$(U^d,h(E^d))$.

\begin{definition}
  Write $H^d$ for the set $\{h(e):e\in E^d\}$.
  The hypergraph $(U^d,H^d)$ is called the \emph{dual hyperlattice}
  or simply the \emph{dual} (of the simplicial lattice).
\end{definition}

\begin{lemma}
  \label{lem:hBijects}
  The map $h:T\mapsto \{h(e):e\in T\}$ is a bijection from $\Theta$ to the set of perfect matchings of the hypergraph $(U^d,H^d)$.
\end{lemma}

Note that $(U^d,H^d)$ is really dual to $(X^d,E^d)$
because the map $h$ is a bijection from $E^d$ to $H^d$.
The hyperlattice is $d!$-uniform, because every hyperedge $h(e)$ contains $d!$ elements.
It is also $d!$-partite with the partition
$\{X^d\times \{\xi\}:\xi\in S_d\}$,
because every hyperedge $h(e)$ contains one loop in each member $X^d\times\{\xi\}$.
The $d!$-partite structure of the dual of the simplicial lattice
is special and it is a feature that distinguishes the simplicial lattice from other lattices
(in particular, the author is not aware of a similar construction for the square lattice in dimension larger than two).
The $d!$-partite structure enables us to generalise the Kasteleyn theory.
% In Section~\ref{sec:Kasteleyn} we show that one can count the perfect matchings of the dual
% with the Cayley hyperdeterminant of the adjacency tensor of the hyperlattice.
% The suggestive terminology has a natural meaning and will be introduced rigorously later.

\subsection{The approach suggested by the classical dimer theory}

% We have introduced four different ways of formulating fixed boundary conditions
% (the first two are given in Definition~\ref{def:fixedbc}; the last two are given by  Lemma~\ref{lem:onemorewayforFBC} and the preceding comment).
The purpose of this section is to demonstrate that the
size of $Z=|\Omega(R,f)|=|\Theta(R,T)|$
equals the Cayley hyperdeterminant of a suitable adjacency hypermatrix.
First recall how this works in the Kasteleyn theory for the dimer model on the hexagonal lattice.
If $d=2$ then $d!=2$,
that is, $(U^d,H^d)$ is a regular bipartite graph.
In fact, it is really the planar dual of the triangular lattice:
the hexagonal lattice.
The vertex set $U^d$ is split into its two parts:
a set of black and a set of white vertices.
A dimer cover (that is, a perfect matching of the graph) matches each black vertex to one white vertex.
This is illustrated by Figure~\ref{fig:2d}\ref{fig:2d:hex}.
The dimer cover is thus encoded by a bijective map $\sigma$ from the set of black vertices
to the set of white vertices;
each dimer is of the form $\{b,\sigma(b)\}$ with $b$
ranging over the set of black vertices.
% The dimers of the dimer cover are thus indexed by the first colour (black in this case);
% each dimer is of the form $\{b,\sigma(b)\}$, where $b$ ranges over the black vertices.
To calculate the number of dimer covers, one needs to count the bijections
$\sigma$ from the black vertices to the white that produce a dimer cover.
If $K$ is an $n\times n$ matrix,
then $\operatorname{Det}K$ is defined as (this is the Leibniz formula)
\begin{equation}
  \label{eq:Leibniz}
  \operatorname{Det}K=\sum_{\sigma\in S_n}\left[\operatorname{Sign}\sigma\prod_{k=1}^nK(k,\sigma(k))\right].
\end{equation}
If the matrix $K$ is suitably chosen,
then the term in the square brackets reduces to the indicator function of the event
that $\sigma$ encodes a dimer cover, in which case $\operatorname{Det}K$ equals
the number of dimer covers.
This is the Kasteleyn theory for dimer models.
These observations suggest the following approach, consisting of four steps:
\begin{enumerate}
  \item \label{encodement} First, encode each tiling as a tuple
  of bijections.
  It turns out that we need $d!-1$ bijections in each tuple,
  because the graph $(U^d,H^d)$ is $d!$-partite, and because we need one bijection for each colour beyond the first.
  This is Lemma~\ref{generalNewBijection}.
  \item \label{finiteEncodement} Second, we show that applying fixed boundary conditions
  fixes the bijections at all but a finite number of points.
  Each tiling $Y\in \Theta(R,T)$ is thus encoded as a
  $(d!-1)$-tuple of bijections between finite sets. This is Lemma~\ref{lem:FBCnewBijection}.
  \item \label{encodementInArray} Third, we define a rank $d!$ adjacency hypermatrix $K$ and construct the Cayley hyperdeterminant
  for this hypermatrix, such that the each nonzero term in the sum in the definition of
  $\operatorname{Det} K$ corresponds to a tiling $Y\in \Theta(R,T)$.
  These are Definitions~\ref{KastArray} and~\ref{genDet}.
  \item \label{encodementSameSign} Finally, each nonzero term in this sum takes value $1$ or $-1$.
  This is due to the signs that appear in the determinant formula
  (note that the sign also appears in~\eqref{eq:Leibniz}).
  It takes some effort to show that all nonzero terms have the same sign.
  This is Lemma~\ref{sameSignLemma}.
\end{enumerate}
Once this has all been done, it is clear that  $Z=|\Omega(R,f)|=|\Theta(R,T)|=\pm\operatorname{Det}K$.
% We heavily rely on the dual hyperlattice $(U^d,H^d)$
% from Section~\ref{sec:simplicial_lattice}; we recall some of the key facts here.
% Recall how we identified loops of $U^d$
% with pairs in $X^d\times S_d$.
% The map $h:E^d\to H^d$ is a bijection that maps each $e\in E^d$
% to the set of unrooted simplicial loops in $U^d$ that cross $e$.
% Each set $h(e)\subset U^d$ contains $d!$ loops,
% and it contains one loop of the form $(\mathbf x,\xi)\in X^d\times S_d$
% for every $\xi\in S_d$.
% Therefore the hypergraph $(U^d,H^d)$
% is $d!$-uniform and $d!$-partite with the partition
% $\{X^d\times \{\xi\}:\xi\in S_d\}$.
The dual hyperlattice $(U^d,H^d)$  plays a crucial role in the analysis.
Fix, throughout the remainder of this section,
an enumeration $\{\xi^1,\dots,\xi^{d!}\}=S_d$.

\subsection{The Kasteleyn theory in dimension $d\geq 2$}
We start with Step~\ref{encodement} of the proposed approach.
Let $Y\in\Theta$ be a tiling of $(X^d,E^d)$,
 so that $h(Y)$ is a perfect matching of $(U^d,H^d)$.
Each hyperedge $h(e)\in h(Y)$ contains one simplicial loop in each of the $d!$
 parts of the partition of $U^d$.
 % Therefore a perfect matching is not encoded by
 % a single bijection (from black vertices to white),
 % but by $d!-1$ bijections (one for each colour beyond the first).
 The bijections corresponding to $Y$ are
 the unique maps
 \begin{equation}
   \label{eqSigmasDef}
   \sigma_i:X^d\times \{\xi^1\}\to X^d\times \{\xi^{i}\}
 \end{equation}
 such that
 \begin{equation}
   \label{eenElementVanSigma}
   \{\mathbf s,\sigma_2(\mathbf s),\sigma_3(\mathbf s),\dots,\sigma_{d!}(\mathbf s)\}\in h(Y)
  \end{equation}
   for every unrooted simplicial loop $\mathbf s\in X^d\times \{\xi^1\}$.
   All elements of $ h(Y)$ are given by ranging $\mathbf s$ over $X^d\times \{\xi^1\}$
   in~\eqref{eenElementVanSigma}.
   This is completely analogous to the dimer model.

 Suppose given arbitrary bijections $(\sigma_i)_{2\leq i\leq d!}$ as in~\eqref{eqSigmasDef}.
 Then the set of sets of simplicial loops
 \begin{equation}
   \label{sigmaPartition}
   \left\{\{\mathbf s,\sigma_2(\mathbf s),\dots,\sigma_{d!}(\mathbf s)\}:\mathbf s\in X^d\times \{\xi^1\}\right\}
 \end{equation}
 is automatically a partition of $U^d=X^d\times S_d$,
 because each map $\sigma_i$ is a bijection
 and therefore each loop $(\mathbf x,\xi^i)$ appears precisely once.
 Conclude that~\eqref{sigmaPartition} is a perfect matching of $(U^d,H^d)$ if and only if~\eqref{sigmaPartition}
 is a subset of the hyperedge set $H^d$.
 This yields the following result: Step~\ref{encodement} of the suggested approach.

  \begin{lemma}
    \label{generalNewBijection}
    The set of $(d!-1)$-tuples of bijections
    \[\left(\sigma_i:X^d\times \{\xi^1\}\to X^d\times \{\xi^i\}\right)_{2\leq i\leq d!}\quad\text{satisfying}
  \quad\{\mathbf s,\sigma_2(\mathbf s),\dots,\sigma_{d!}(\mathbf s)\}\in H^d\]
    for every loop $\mathbf s\in X^d\times \{\xi^1\}$
    is in bijection with the perfect matchings of $(U^d,H^d)$.
    The perfect matching of a tuple (under this bijection) is given by ranging $\mathbf s$ over $X^d\times \{\xi^1\}$;
    this is precisely the set in~\eqref{sigmaPartition}.
  \end{lemma}

Now consider Step~\ref{finiteEncodement} of the suggested approach.
Suppose given a region $R$ and a tiling $T$,
and consider a tiling $Y\in \Theta(R,T)$.
  By definition, $Y\in \Theta(R,T)$ if and only if
  $Y\smallsetminus E^d(R)=T\smallsetminus E^d(R)$.
  Therefore all loops traversing $T\smallsetminus E^d(R)$
  must be matched in the same way as in $T$,
  and the loops traversing $T\cap E^d(R)$
  can be matched differently.
  However,
  the loops that are matched differently are not
  allowed to produce new hyperedges outside the set $h(E^d(R))$,
  since we want $Y\smallsetminus E^d(R)=T\smallsetminus E^d(R)$.
  We first need to identify,
  for each part of the partition
  $\{X^d\times \{\xi^i\}:1\leq i\leq d!\}$
  of $U^d$,
  the set of loops traversing
  $T\cap E^d(R)$, that is, the loops that are allowed to be matched differently.
  This motivates the following definition.

  \begin{definition}
    Define, for a fixed region $R$ and a fixed tiling $T$,
    and for $1\leq i\leq d!$,
    \begin{align*}
      X_i&:=\{(\mathbf x,\xi^i):\text{the loop $(\mathbf x,\xi^i)$ intersects $T\cap E^d(R)$}\}\\
      &\phantom{:}=\{(\mathbf x,\xi^i):\text{the loop $(\mathbf x,\xi^i)$ does not intersect $T\smallsetminus E^d(R)$}\}\subset X^d\times \{\xi^i\}.
    \end{align*}
  \end{definition}

  Observe that $|X_i|=|T\cap E^d(R)|$,
  and therefore the sets $X_i$ all have the same, finite size.
  The sets $X_i$ contain the loops that are allowed to match differently.
  This is Step~\ref{finiteEncodement} of the suggested approach.

  \begin{lemma}
    \label{lem:FBCnewBijection}
    Let $R$ be a region and $T$ a tiling.
    The set of $(d!-1)$-tuples of bijections \[\left(\sigma_i:X_1\to X_i\right)_{2\leq i\leq d!}\qquad\text{satisfying}\qquad
    \{\mathbf s,\sigma_2(\mathbf s),\dots,\sigma_{d!}(\mathbf s)\}\in h(E^d(R))\]
    for every loop $\mathbf s\in X_1$
    is in bijection with the set of perfect matchings $h(Y)$ corresponding to tilings
     $Y\in\Theta(R,T)$.
    The perfect matching of a tuple (under this bijection) is
    \[\textstyle\left\{\{\mathbf s,\sigma_2(\mathbf s),\dots,\sigma_{d!}(\mathbf s)\}:\mathbf s\in X_1\middle\}\cup \middle \{h(e):e\in T\smallsetminus E^d(R)\right\}.\]
  \end{lemma}

  The Kasteleyn hypermatrix and its determinant are now straightforwardly defined.

  \begin{definition}\label{KastArray}
    Let $R$ be a region and $T$ a tiling.
    Define
    \[K:X_1\times \dots\times X_{d!}\to\{0,1\},\,(\mathbf s_1,\dots,\mathbf s_{d!})\mapsto 1\left(\{\mathbf s_1,\dots,\mathbf s_{d!}\}\in h(E^d(R))\right),\]
    where $1(\cdot )$ equals one if the statement inside holds true and zero otherwise.
    The map $K$ is called the \emph{Kasteleyn hypermatrix}.
  \end{definition}

  From this definition and the previous lemma it follows that
  \[|\Theta(R,T)|=\sum_{\sigma_2:X_1\to X_2,\dots,\sigma_{d!}:X_1\to X_{d!}}\left[\prod_{\mathbf s\in X_1}K(\mathbf s,\sigma_2(\mathbf s),\dots,\sigma_{d!}(\mathbf s))\right],\]
  where the sum is over bijective maps only.
  This because the product produces a $1$ if the tuple $(\sigma_i)_{2\leq i\leq d!}$ corresponds to an element of $\Theta(R,T)$
  and zero otherwise.
  Recall that $|X_1|=\dots=|X_{d!}|=|T\cap E^d(R)|$ and write $n$ for this finite number.
  To simplify notation we identify each set $X_i$ with $[n]:=\{1,\dots,n\}$,
  so that the previous equality is written
  \begin{equation}
    \label{firstDetLikeForm}
    |\Theta(R,T)|=\sum_{\sigma_2,\dots,\sigma_{d!}\in S_n}\left[\prod_{k=1}^n K(k,\sigma_2(k),\dots,\sigma_{d!}(k))\right].
  \end{equation}
  The expression on the right looks similar to the definition of the determinant of a matrix,
  and if we insert the signs of the permutations then we obtain precisely the Cayley hyperdeterminant.

  \begin{definition}
    \label{genDet}
    Suppose given a map $A:[n]^m\to\mathbb C$
    for some $n\in \mathbb N$, $m\in2\mathbb N$.
    Define
    \begin{alignat}{2}\label{firstDetDef}
      \operatorname{Det}A&:=&&\sum_{\sigma_2,\dots,\sigma_m\in S_n}
      \left(
      \left[\prod_{i=2}^m \operatorname{Sign}\sigma_i\right]\left[\prod_{k=1}^n A(\phantom{\sigma_1(}k\phantom{)},\sigma_2(k),\dots,\sigma_m(k))\right]
      \right)
    \\
      \label{bigDetExpans}
      &\phantom{:}=\frac{1}{n!}&&\sum_{\sigma_1,\dots,\sigma_m\in S_n}
      \left(
      \left[\prod_{i=1}^m \operatorname{Sign}\sigma_i\right]\left[\prod_{k=1}^n A(\sigma_1(k),\sigma_2(k),\dots,\sigma_m(k))\right]
      \right).
    \end{alignat}
    This expression is called the \emph{Cayley hyperdeterminant} of $A$.
    %To see that the equality holds,
    %note that each term in the first sum appears $n!$ times in the second sum,
    %always with the same sign, because $m$ is even.
  \end{definition}
  The equality in the definition is straightforwardly verified, and it requires $m$ to be even.
  If we replace $A$ by $K$ in~\eqref{firstDetDef}
  then~\eqref{firstDetLikeForm} and~\eqref{firstDetDef} are the same,
  except that some signs appear in~\eqref{firstDetDef}
  that do not appear in~\eqref{firstDetLikeForm}.
  We conclude that the nonzero terms of the sum in~\eqref{firstDetDef} correspond
  precisely to the elements of $\Theta(R,T)$.
  This is Step~\ref{encodementInArray} of the proposed approach.
  In order to prove that
  $|\Theta(R,T)|=\pm \operatorname{Det}K$,
  it suffices to show that all terms of the sum in the definition of
  $\operatorname{Det}K$ have the same sign (this is Step~\ref{encodementSameSign}).

  \begin{lemma}
    \label{sameSignLemma}
    Let $R$ be a region and let $T$ be a tiling.
    Write $K$ for the Kasteleyn hypermatrix.
    Then all nonzero terms in the sum in the definition of
    $\operatorname{Det}K$ have the same sign.
  \end{lemma}

  \begin{proof}
    Let $R$, $T$ and $K$ be as in the lemma.
    We want to show that all terms of the sum in~\eqref{firstDetDef} (with $A$ replaced with $K$)
    have the same sign.
    The idea is to show that the sign is invariant under making a local move as defined in
    Subsection~\ref{subsec:localmove}.
    Write $f$ for the unique height function such that $\Phi(f)=(0,T)$.
    The nonzero terms in~\eqref{firstDetDef} correspond bijectively (through the bijections that we have set up
    in
    Lemma~\ref{lem:FBClemma},~\ref{lem:fbc_bijection}
    and in Lemma~\ref{lem:FBCnewBijection})
    to the height functions in $\Omega(R,f)$.
    We pick two height functions
    $f',f''\in \Omega(R,f)$ and prove that the corresponding terms in~\eqref{firstDetDef} have the same sign.
    By Lemma~\ref{lem:localMove},
    we may assume, without loss of generality,
    that $f''=f'+(d+1)\cdot 1_\mathbf x$ for some $\mathbf x\in R$.
    Let $T',T''\in\Theta(R,T)$ be the tilings corresponding to $f',f''$ respectively.
    Recall that
    \[T'=\left\{\{\mathbf y,\mathbf y+\mathbf g_i\}\in E^d:\nabla f'(\mathbf y,\mathbf y+\mathbf g_i)=-d\right\},\]
    and for $T''$ we have an identical expression in terms of $f''$.
    Remark that $f''=f'$ except at the point $\mathbf x$,
    and therefore $\nabla f=\nabla f'$ except at the edges incident to $\mathbf x$.
    Since $f''=f'+(d+1)\cdot 1_\mathbf x$ and since both $f'$ and $f''$ are height functions,
    we must have
    \begin{align*}
    &\nabla f'(\mathbf x,\mathbf x+\mathbf g_i)=\nabla f''(\mathbf x-\mathbf g_i,\mathbf x)=1,
    \\
    &\nabla f'(\mathbf x-\mathbf g_i,\mathbf x)=\nabla f''(\mathbf x,\mathbf x+\mathbf g_i)=-d,
  \end{align*}
    and therefore
    \begin{equation}
      \label{defNenNApostr}
      \begin{split}
        &T'\smallsetminus T'' = \{\{\mathbf x,\mathbf x-\mathbf g_i\}:1\leq i\leq d+1\},\\
      &  T''\smallsetminus T' = \{\{\mathbf x,\mathbf x+\mathbf g_i\}:1\leq i\leq d+1\}.
      \end{split}
    \end{equation}
    This means that all loops are matched the same (in the matchings $h(T')$ and $h(T'')$),
    except for the loops traversing the vertex $\mathbf x$.
    In order to prove the lemma, we work out the effect of this difference on the signs in~\eqref{firstDetDef}.
    Let $(\sigma_i')_{2\leq i\leq d!}$ denote the bijections from Lemma~\ref{lem:FBCnewBijection}
    corresponding to $T'$.
    This means that
    \[\{\{\mathbf s,\sigma_2'(\mathbf s),\dots,\sigma_{d!}'(\mathbf s)\}:\mathbf s\in X_1\}=h(T'\cap E^d(R)).\]
    Define, for each $1\leq i\leq d!$, the bijections
    \[\delta_i:X_i\to X_i,\,(\mathbf x,\xi^i)\mapsto \begin{cases}
    (\mathbf x+\mathbf g_j,\xi^i)&\text{if $(\mathbf x,\xi^i)$ traverses $\{\mathbf x-\mathbf g_j,\mathbf x\}$ for some $j$},\\
    (\mathbf x,\xi^i)&\text{if $(\mathbf x,\xi^i)$ does not traverse $\mathbf x$}.
    \end{cases}\]
    Note that $\delta_i$ is a permutation consisting of one cycle of length $d+1$.
    Claim that
    \begin{equation}
      \label{eq:setofbriefinterest}
    \{\{\delta_1(\mathbf s),\delta_2\circ \sigma_2'(\mathbf s),\dots,\delta_{d!}\circ \sigma_{d!}'(\mathbf s)\}:\mathbf s\in X_1\}=h(T''\cap E^d(R)).
    \end{equation}
    To support the claim,
    recall~\eqref{defNenNApostr}
    and
    observe simply that
    \begin{multline*}
      \{\delta_1(\mathbf s),\delta_2\circ \sigma_2'(\mathbf s),\dots,\delta_{d!}\circ \sigma_{d!}'(\mathbf s)\}\\
      =\begin{cases}
    h(\{\mathbf x,\mathbf x+\mathbf g_i\})&\text{if $\{\mathbf s, \sigma_2'(\mathbf s),\dots, \sigma_{d!}'(\mathbf s)\}=h(\{\mathbf x,\mathbf x-\mathbf g_i\})$ for some $i$},\\
    \{\mathbf s, \sigma_2'(\mathbf s),\dots, \sigma_{d!}'(\mathbf s)\}&\text{otherwise}.
    \end{cases}
  \end{multline*}
  This proves the claim.
  Since $\delta_1:X_1\to X_1$ is a bijection,
  the sets in~\eqref{eq:setofbriefinterest}
  are equal to
  \[
    \{\{\mathbf s,\delta_2\circ \sigma_2'\circ\delta_1^{-1}(\mathbf s),\dots,\delta_{d!}\circ \sigma_{d!}'\circ\delta_1^{-1}(\mathbf s)\}:\mathbf s\in X_1\}
  \]
    This implies that
    the bijections from Lemma~\ref{lem:FBCnewBijection}
    corresponding to $T''$ are, for $2\leq i\leq d !$,
    \[\sigma_i''=\delta_i\circ \sigma_i'\circ \delta_1^{-1}:X_1\to X_i.\]
    Now note that $\operatorname{Sign}\delta_i=(-1)^{d}$
    (since $\delta_i$ is a cycle of length $d+1$).
    Conclude that $\operatorname{Sign}\delta_i \cdot \operatorname{Sign}\delta_1^{-1}=1$,
    and therefore $\sigma_i'$ and $\sigma_i''$ have the same sign in~\eqref{firstDetDef}, for all $i$.
  \end{proof}

  We have completed the final step of the approach that was suggested by the Kasteleyn
  theory for dimer covers. This yields the following theorem.

  \begin{theorem}
    \label{thm:KAST}
    Let $R$ be a region, let $f$ be a height function, and let $T=T(f)$.
    Write $K$ for the Kasteleyn hypermatrix.
    Then $Z=|\Omega(R,f)|=|\Theta(R,T)|=\pm\operatorname{Det}K$.
  \end{theorem}

\subsection{Boltzmann measures}
  Recall the definition of a Boltzmann measure in Subsection~\ref{subsec:fbc}.
  The number $|\Omega(R,f)|=|\Theta(R,T)|$ equals the partition function $Z$ of the uniform probability measures
  on $\Omega(R,f)$ and $\Theta(R,T)$.
  The Kasteleyn theory is easily generalised to Boltzmann measures
  by inserting the weights into the
  Kasteleyn hypermatrix.

  \begin{definition}
    Let $R$ be a region, $T$ a tiling, and $w:E^d(R)\to \mathbb C$ any (complex-valued) weight function.
    Define
    \begin{align*}
    K_w:{}&X_1\times \dots\times X_{d!}\to \mathbb C,\\
    &(\mathbf s_1,\dots,\mathbf s_{d!})\mapsto 1\left(\{\mathbf s_1,\dots,\mathbf s_{d!}\}\in h(E^d(R))\right)\cdot w\left(h^{-1}(\{\mathbf s_1,\dots,\mathbf s_{d!}\})\right).
  \end{align*}
    The map $K_w$ is called the \emph{weighted Kasteleyn hypermatrix}.
  \end{definition}

  By comparing the definition of $Z_w$ with the definition of the Cayley hyperdeterminant,
  and taking into account Lemma~\ref{sameSignLemma}, it is readily verified
  that $Z_w=\pm\operatorname{Det}K_w$.

\subsection{Explicit calculations}

Two natural questions arise, which unfortunately we are not able to answer.
First, one may ask if it is possible to find a simple expression for the partition function in the
case of the torus,
in order to derive a closed-form formula for the surface tension.
This does not appear to be the case.
The Cayley hyperdeterminant is invariant under a change of basis,
and it is therefore meaningful to perform a Fourier transform on $K$.
The transformed hypermatrix is ``diagonal'',
in the sense that its rank-$d!$ hyperdeterminant
reduces to the permanent of a hypermatrix of rank $d!-1$.
This leads to the desired formula only in dimension two.
Second, one may ask if it is possible to derive results on dimer-dimer correlations.
The absence of a natural inverse for $K$ leads the author to believe
that also this second question is more difficult in the general case.

%!TEX root= ../main.tex

\section{Gradient Gibbs measures}
\label{sec:gibbs}

%In previous sections we only worked with
%probability measures on finite subsets of $\Omega$.
%In this section we introduce measures
%on the entire set $\Omega$.
%The interest is in so-called Gibbs measures,
%which are invariant under resampling the function values
%of the drawn height function on finite subsets of $X^d$.
%This comes at a technical cost: it requires the introduction of $\sigma$-algebras
%and a proof of existence of the desired measures.
%This is a straightforward operation in the gradient setting.
%There is also an important benefit: there exist shift-invariant gradient Gibbs measures on $\Omega$,
%and we have an expression for the surface tension $\sigma$ in terms of such measures.
%This eventually leads to a proof of strict convexity in the next section.

In previous sections we introduced fixed boundary conditions
and periodic boundary conditions, which enabled us to study probability measures
on finite subsets of $\Omega$.
This section introduces \emph{shift-invariant gradient Gibbs measures},
which prove to be an effective tool for studying the large-scale behaviour of the model.
While gradient Gibbs measures are interesting in their own right,
their main purpose here are their use in the proof of strict convexity of
the surface tension in Section~\ref{sec:strict}.

\subsection{Definition}

Write $f$ for the random function in $\Omega$.
Define for any $R\subset X^d$,
\begin{alignat*}{4}
  &\mathcal F&&:=\sigma(f(\mathbf x):\mathbf x\in X^d),
    &&\mathcal F_R&&:=\sigma(f(\mathbf x):\mathbf x\in R),\\
  &\mathcal F^\nabla&&:=\sigma(f(\mathbf x)-f(\mathbf y):\mathbf x,\mathbf y\in X^d),
    \qquad
    &&\mathcal F^\nabla_R&&:=\sigma(f(\mathbf x)-f(\mathbf y):\mathbf x,\mathbf y\in R).
\end{alignat*}
Note that $\mathcal F^\nabla_R=\mathcal F^\nabla \cap \mathcal F_R$ is finite whenever $R$ is finite
because it is generated by finitely many random variables, each taking finitely many values.
Write $\mathcal P(\Omega,\mathcal X)$ for the collection of probability measures
on the measurable space $(\Omega,\mathcal X)$ for any $\sigma$-algebra $\mathcal X$ on $\Omega$.
Probability measures in $\mathcal P(\Omega,\mathcal F^\nabla)$ are called \emph{gradient measures}.

A gradient measure $\mu\in\mathcal P(\Omega,\mathcal F^\nabla)$
is called \emph{shift-invariant} whenever $\mu(\tilde\theta A)=\mu(A)$
for any $A\in\mathcal F^\nabla$ and $\theta\in\Theta$,
where $\tilde\theta A:=\{\tilde\theta f:f\in A\}$.
In other words, a gradient measure $\mu$
is shift-invariant whenever $\nabla f$ and $\theta\nabla f$
have the same law in $\mu$ for every $\theta\in\Theta$.
The set of shift-invariant gradient measures
is denoted by $\mathcal P_\Theta(\Omega,\mathcal F^\nabla)$.
If  $\mu\in \mathcal P_\Theta(\Omega,\mathcal F^\nabla)$,
then it follows from shift-invariance that
the map $\mu(f(\cdot)-f(\mathbf 0)):X^d\to\mathbb R$
is additive over $X^d$.
Therefore there exists a unique $s\in H^*$
such that
$s(\mathbf x)= \mu(f(\mathbf  x)-f(\mathbf 0))$
for every $\mathbf x\in X^d$,
and we must have $s\in\mathcal S$ because $s(\mathbf g_i)=\mu(f(\mathbf g_i)-f(\mathbf 0))\leq 1$ for every $1\leq i\leq d+1$.
Write $s(\mu)$ for $s$, the \emph{slope} of $\mu\in\mathcal P_\Theta(\Omega,\mathcal F^\nabla)$.

Let $(L,s)$ denote valid periodic boundary conditions
and let $\mu$ denote the probability measure that is uniformly random in the finite set $\Omega(L,s)$.
Lemma~\ref{lem:expectation_periodic}
implies
that $\mu$ restricts to a shift-invariant gradient measure
of slope $s(\mu)=s$.

Let us now introduce the notion of a Gibbs measure.
Fix a measure $\mu\in\mathcal P(\Omega,\mathcal F)$.
The measure $\mu$ is called a \emph{Gibbs measure}
if for every finite $R\subset X^d$,
the distribution of $f$ in $\mu$ is the same
as the distribution of a sample $f$
obtained by
first sampling $g$ from $\mu$,
then sampling $f$ from $\Omega(R,g)$
uniformly at random.
The definition is formalised in terms of specifications and the Dobrushin-Lanford-Ruelle (DLR) equations.
For each finite $R\subset X^d$, let $\gamma_R$ denote
the probability kernel from $(\Omega,\mathcal F_{R^\complement})$ to $(\Omega,\mathcal F)$
such that for any $f\in\Omega$, the probability measure $\gamma_R(\cdot,f)$ is uniform in $\Omega(R,f)$.
It is obvious from the definition that $\Omega(R,f)$ is invariant under changing the values of
$f$ on $R$,
so that $\gamma_R(A,\cdot)$ is indeed $\mathcal F_{R^\complement}$-measurable for every $A\in\mathcal F$.
% Thus, if $\mu\in\mathcal P(\Omega,\mathcal F)$,
% then $\mu\gamma_R$ is another probability measure on $\mathcal P(\Omega,\mathcal F)$;
% to sample from $\mu\gamma_R$, first sample a height function $g$ from $\mu$,
% then draw the final sample uniformly at random from the set $\Omega(R,g)$.
The kernels $\gamma_R$ satisfy the consistency condition;
if $S\subset R$, then
$\gamma_R\gamma_{S}=\gamma_R$.
The collection of probability kernels $\gamma_R$ is called a specification,
and a measure $\mu\in\mathcal P(\Omega,\mathcal F)$
is called a \emph{Gibbs measure} if $\mu$ satisfies the DLR equation
\begin{equation}
  \label{eq:DLR}
  \mu=\mu\gamma_R
\end{equation}
for each finite $R\subset X^d$.
This is equivalent to our previous, informal description.
By the consistency condition it is sufficient
to check the DLR equations for an increasing exhaustive sequence
of finite subsets of $X^d$.
%
%Fix a height function $f$ and a finite set $R\subset X^d$.
%The measures in Equation~\ref{eq:DLR} are measures on the $\sigma$-algebra $\mathcal F$.
%Suppose now that we are interested in a smaller $\sigma$-algebra.
%What is the minimum amount of information that we should provide to
%the probability kernel $\gamma_R$ for the resampling procedure to make sense?
%The answer depends crucially on Remark~\ref{rem:trivialIntermediate},~\ref{secondpoint}.
%First, the remark says that to resample $f$ on $R$,
%all that we need to know are the values of $f$ on $\partial R$.
%If $S\subset X^d$ contains $R\cup\partial R$,
%then $\gamma_R$ restricts to a probability kernel from $(\Omega,\mathcal F_{S\smallsetminus R})$
%to $(\Omega,\mathcal F_S)$.
%(Clearly we cannot know the values of $f$ on $S^\complement$
%after resampling if we did not know these before resampling.)
%Second, the remark says that for the resampling it is sufficient to know the values
%of $f$ up to an additive constant.
Each kernel $\gamma_R$ restricts to a kernel from
$(\Omega,\mathcal F^\nabla_{R^\complement})$ to $(\Omega,\mathcal F^\nabla)$.
We shall write $\gamma_R^\nabla$ for this restriction.
%We may also combine these two facts:
%the kernel $\gamma_R$ restricts to a probability kernel from
%$(\Omega,\mathcal F_{S\smallsetminus R}^\nabla)$ to $(\Omega,\mathcal F_S^\nabla)$
%whenever $S\subset X^d$ and contains $R\cup\partial R$.
%
A gradient measure $\mu\in\mathcal P(\Omega,\mathcal F^\nabla)$ is called a
\emph{gradient Gibbs measure} if
\[\mu=\mu\gamma_R^\nabla\]
for each finite subset $R$ of $X^d$.

% \begin{remark}
%   \label{remark:whyGradient?}
% We only define shift-invariance in the gradient setting.
% The reason for that is as follows.
% A gradient measure is shift-invariant
% if the distribution of $\nabla f$,
% the discrete derivative of the random function,
% is invariant under shifts in $\Theta$.
% A gradient measure is defined with respect to
% the smallest $\sigma$-algebra that makes $\nabla f$
%  measurable.
% A non-gradient measure records in addition the value of $f$
% at one vertex---and therefore the value of $f$ at every vertex.
% How do we deal with this extra information when applying a shift to a height function?
% One answer to this question is the original definition of $\theta f$
% in Equation~\ref{eq:definitionThetaF};
% that definition simply ensures that the shift $\theta$ does not alter
% the value of $f$ at the vertex $\mathbf 0$.
% This worked well in Section~\ref{section:PLSD}, because all height functions in
% $\Omega(L,s)$ map $\mathbf 0$ to $0$---by definition.
% The same choice however becomes entirely arbitrary
% when not enforcing (periodic) boundary conditions.
% The author believes that no canonical choice
% exists that does justice to the symmetry of the general translation-invariant setting.
% Working in the gradient setting
% and forgetting about the extra bit of information circumvents the problem altogether, and is therefore more natural.
% \end{remark}

\subsection{Existence and concentration}

\begin{theorem}
  \label{thm:existence_big_Gibbs_measures}
  For each slope $s\in\mathcal S$,
  there is a shift-invariant gradient Gibbs measure $\mu\in\mathcal P_\Theta(\Omega,\mathcal F^\nabla)$
  of slope $s$ such that, for any $\mathbf x,\mathbf y\in X^d$,
  we have the bounds
  \begin{enumerate}
    \item
    \label{thm:shift_inv:one}
    $\operatorname{Var}_\mu (f(\mathbf y)-f(\mathbf x))\leq (d+1)^2 n$,
    \item $\mu(f(\mathbf y)-f(\mathbf x)-s(\mathbf y-\mathbf x)\geq (d+1)a )\leq \exp -\frac{a^2}{2n}$ for all $a\geq 0$ whenever $n>0$,
    \item
    \label{thm:shift_inv:three}
    $\mu(f(\mathbf y)-f(\mathbf x)-s(\mathbf y-\mathbf x)\leq (d+1)a )\leq \exp -\frac{a^2}{2n}$ for all $a\leq 0$ whenever $n>0$,
  \end{enumerate}
  where $n=d_{(X^d,E^d)}(\mathbf x,\mathbf y)$.
\end{theorem}

The \emph{topology of local convergence} or \emph{$\mathcal L$-topology} on $\mathcal P(\Omega,\mathcal X)$
is the coarsest topology that makes the evaluation map $\mu\mapsto\mu(A)$
continuous for every finite $R\subset X^d$ and for any $A\in\mathcal X\cap\mathcal F_R$.
Constructing (gradient) Gibbs measures on $\mathcal P(\Omega,\mathcal X)$ is much easier whenever choosing $\mathcal X=\mathcal F^\nabla$
and not $\mathcal X=\mathcal F$, because $\mathcal F_R^\nabla$ is finite for finite $R\subset X^d$---see the following
lemma.

\begin{lemma}
  The set $\mathcal P(\Omega,\mathcal F^\nabla)$ is compact in the topology
   of local convergence.
\end{lemma}

\begin{proof}
  The proof is entirely straightforward.
  Let $(\mu_n)_{n\in\mathbb N}$ denote a sequence of measures in $\mathcal P(\Omega,\mathcal F^\nabla)$ and let
   $(\Gamma_m)_{m\in\mathbb N}$ denote an increasing exhaustive sequence
  of finite subsets of $X^d$.
  Fix $m\in\mathbb N$.
  The $\sigma$-algebra
   $\mathcal F_{\Gamma_m}^\nabla$ is finite and therefore there exists a subsequence $(k_n)_{n\in\mathbb N}\subset \mathbb N$
  such that $\mu_{k_n}$ converges on  $\mathcal F_{\Gamma_m}^\nabla$   as $n\to\infty$.
  By a standard diagonalisation argument we may assume that convergence
 occurs for all $m\in\mathbb N$.
  The limiting measure exists by the Kolmogorov extension theorem.
\end{proof}

% The measure in Theorem~\ref{thm:existence_big_Gibbs_measures}
% is constructed as follows.
% We start with a sequence of shift-invariant
% gradient measures $(\mu_n)_{n\in\mathbb N}$ with the additional property that
% $\nabla f$ is $n\Theta$-periodic almost surely in $\mu_n$.
% Each measure $\mu_n$ is chosen such that it satisfies a weaker
% form of the DLR equations.
% It is shown that each measure satisfies the desired variance bound
% by an analysis similar to Section~\ref{sec:varcovar}.

\begin{proof}[Proof of Theorem~\ref{thm:existence_big_Gibbs_measures}]
  Let $s\in\mathcal S$ be the slope of interest.
  Let $(s_n)_{n\in\mathbb N}$ be a sequence of slopes
  converging to $s$ with $s_n\in \mathcal S_n$ for every $n$.
  Write $\mu_n$ for the uniform probability measure on
  $\Omega(L_n,s_n)$, for every $n\in\mathbb N$.
  Each measure $\mu_n$ restricts to a shift-invariant gradient measure
  in $\mathcal P_\Theta(\Omega,\mathcal F^\nabla)$,
  and $s(\mu_n)=s_n$.

  Now apply the previous lemma to obtain a subsequence $(k_n)_{n\in\mathbb N}$
  along which the sequence of gradient measures $(\mu_n)_{n\in\mathbb N}$ converges in the topology of local convergence,
  say to $\mu\in\mathcal P(\Omega,\mathcal F^\nabla)$.
  The limit $\mu$ must be shift-invariant as all measures $(\mu_n)_{n\in\mathbb N}$ are
  shift-invariant.
  At each vertex $\mathbf x\in X^d$ we have
  \[
    \mu(f(\mathbf x)-f(\mathbf 0))
    =\lim_{n\to\infty}\mu_{k_n}(f(\mathbf x)-f(\mathbf 0))
    =\lim_{n\to\infty}s_{k_n}(\mathbf x)
    =s(\mathbf x),\]
  which means that $s(\mu)=s$.
  One shows similarly that \eqref{thm:shift_inv:one}--\eqref{thm:shift_inv:three} follow
  from Theorem~\ref{theorem:Azuma}.

  It suffices to prove that the gradient measure $\mu$ is a Gibbs measure,
  that is,
  that
  $\mu\gamma_R^\nabla=\mu$ for every finite $R\subset X^d$.
  Fix a finite subset $R\subset X^d$.
  Now suppose that $\mu\gamma_R^\nabla$ equals $\mu$
  on $\mathcal F_S^\nabla$
  for any finite $S\subset X^d$.
  Then the two measures must be the same, by
  the uniqueness statement of the Kolmogorov extension theorem.
  It thus suffices to prove that $\mu\gamma_R^\nabla$ equals $\mu$
  on $\mathcal F_S^\nabla$
  for any finite $S\subset X^d$.
  We may assume that $R\subset S$ and $\partial R\subset S$ by expanding $S$ if necessary.
  By using shift-invariance,
  we may finally assume that $\0\not\in S$.
%  By shift-invariance of $\mu$
%  we may furthermore suppose that
%  $S\subset B_m$ for some
%  $m\in \mathbb N$,
%  by shifting $R$ and $S$ by some $\theta\in\Theta$
%  if necessary.
%  Finally, we expand $S$ again so that $S=B_m$.

  We make the stronger claim that already in the non-gradient setting and before taking limits, we have
  \begin{equation}
    \label{eq:nonGradientStrongerClaim}
    \mu_n\gamma_R|_{\mathcal F_S}=\mu_n|_{\mathcal F_S}
  \end{equation}
  for $n$ sufficiently large.
    The distribution $\mu_n$ is not invariant under resampling $f$
  on $R$.
  However, if $R+\x$ and $R+\y$
  are disjoint and not adjacent for
  any $\x,\y\in L_n$ distinct,
  then $\mu_n$ is invariant under resampling $f$
  on $R$, then translating
  this modification to $R+\x$
  for each $\x\in L_n\smallsetminus\{\0\}$.
  Thus, if $n$ is so large that $S$
  and $R+\x$ are disjoint for any $\x\in L_n\smallsetminus\{\0\}$,
  then~\eqref{eq:nonGradientStrongerClaim}
  holds true because the additional modifications
  do not affect the values of $\mu_n$ on the $\sigma$-algebra $\mathcal F_S$.
  This proves the claim.
\end{proof}

For any $n\in\mathbb N$, let $\Pi_n$ denote a \emph{centred box} of sides $2n$,
that is,
\[\Pi_n:=\{a_1\mathbf g_1+\dots+a_d\mathbf g_d:-n\leq a_1,\dots,a_d<n\}\subset X^d.\]
Note that $|\Pi_n|=(2n)^d$.
\label{eq:pi_n_def}

\begin{proposition}
  \label{propo:limitFlatness}
  Let $\mu$ denote a measure of Theorem~\ref{thm:existence_big_Gibbs_measures} of slope $s\in\mathcal S$.
   Then $\mu$-almost surely
  \begin{equation}
    \label{eq:flat_conc}
    \lim_{n\to\infty}\frac 1n \| (f-f(\mathbf 0))|_{\Pi_n}-s|_{\Pi_n}\|_\infty = 0.
    \end{equation}
\end{proposition}
This follows from a union bound and the inequalities in Theorem~\ref{thm:existence_big_Gibbs_measures}.

%!TEX root= ../main.tex
\section{The surface tension and the variational principle}
%\section{The large deviations principle and the variational principle}
\label{sec:Statement of the variational principle}

The purpose of this section is to give an overview
of three closely related concepts
which describe the macroscopic behaviour of the model.
These motivate the study of strict convexity
of the surface tension in Section~\ref{sec:strict}.
First, there is indeed the surface tension,
which describes the asymptotic number of height
functions approximating a certain slope.
Second, there is the large deviations principle,
which describes the asymptotic number of height
functions approximating an arbitrary continuous profile.
The rate function is the integral of the surface
tension over the gradient of the continuous profile of interest.
Third, there is the variational principle,
which is a direct corollary of the large deviations principle,
and describes the typical macroscopic behaviour of the random height function.
The surface tension is usually convex,
making that the rate function in the large deviations principle
is also convex.
In the next section, we shall also prove that the surface tension
is strictly convex,
which implies that the rate function has a unique minimiser,
which in turn implies concentration around a single continuous
profile in the variational principle.
For the results in this section, we refer to~\cite{SHEFFIELD} and~\cite{MINIMAL}.

\subsection{The surface tension}

\begin{definition}
    Recall the definition of $\lfloor f\rfloor$
    for Lipschitz functions $f:H\to\R$
    on Page~\pageref{pagerefLipfloor}.
    Recall also the definition of $\Pi_n\subset X^d$
    at the end of the previous section (Page~\pageref{eq:pi_n_def}).
    The \emph{surface tension} is the function $\sigma:\mathcal S\to\R$ defined by
    \[
        \sigma(s):=\lim_{n\to\infty}-\frac{1}{|\Pi_n|}\log|\Omega(\Pi_n,\lfloor s\rfloor)|.
    \]
\end{definition}

For convergence of the limit in the definition of $\sigma(s)$,
we refer to Section~4 in~\cite{MINIMAL}.
The argument is effectively a supermultiplicativity argument:
if $A,B\subset X^d$ are finite and disjoint with no vertex of $A$ adjacent
to a vertex of $B$,
then
$|\Omega(A\cup B,f)|=|\Omega(A,f)|\cdot|\Omega(B,f)|$,
and if $A\subset B$,
then $|\Omega(A,f)|\leq |\Omega(B,f)|$.
In fact, the definition of $\sigma(s)$
is stable under modifications of order $o(n)$
to $\lfloor s\rfloor$ as $n\to\infty$;
see Lemma~4.5 in~\cite{MINIMAL} for the following result.

\begin{theorem}
    \label{thm:conv_to_sigma_stable}
    If $s\in\mathcal S$, and if $(f_n)_{n\in\mathbb N}\subset \Omega$
    satisfies
    $\|f_n|_{\Pi_n}-s|_{\Pi_n}\|_\infty=o(n)$
    as $n\to\infty$,
    then
    \[
        \sigma(s)=\lim_{n\to\infty}-\frac1{|\Pi_n|}\log|\Omega(\Pi_n,f_n)|.
    \]
\end{theorem}

The previous result implies immediately that $\sigma$ is continuous,
see also Lemma~4.3 in~\cite{MINIMAL}.

\begin{theorem}
   The surface tension $\sigma:\mathcal S\to\mathbb R$ is continuous.
\end{theorem}

\subsection{The large deviations principle}

Before stating the large deviations principle,
we must introduce a suitable topological space to work in,
and we must specify how a sequence of fixed boundary conditions converges to a
continuous boundary profile.
This is the purpose of the following definition.

\begin{definition}
    Write $\operatorname{Lip}(D)$
    for the collection of real-valued Lipschitz functions
    on $D$ for any $D\subset H$.
    A \emph{domain} is a bounded open set $D\subset H$
    such that $\partial D$ has zero Lebesgue measure.
    A \emph{boundary profile} is a pair $(D,b)$
    where $D$ is a domain and $b\in\operatorname{Lip}(\partial D)$.
    An \emph{approximation} of $(D,b)$
    is a sequence of pairs $((D_n,b_n))_{n\in\mathbb N}$
    such that $D_n\subset X^d$ is finite and $b_n\in\Omega$
    for any $n\in\mathbb N$,
    and such that
    \[
        \frac1n D_n\to D,\qquad \frac1n\operatorname{Graph}b_n|_{\partial D_n}\to\operatorname{Graph}b
    \]
    in the Hausdorff topologies on $H$ and $H\times\mathbb R$
    respectively as $n\to\infty$.

    If $(D,b)$ is a boundary profile with approximation $((D_n,b_n))_{n\in\mathbb N}$,
    then write $(\gamma_n)_{n\in\mathbb N}$
    for the sequence of measures
    defined by $\gamma_n:=\gamma_{D_n}(\cdot,b_n)$,
    the uniform probability measure in the finite set $\Omega(D_n,b_n)$.
    The topological space for the large deviations principle associated
    with this sequence is the set $\operatorname{Lip}(\bar D)$ endowed
    with the topology of uniform convergence---which is equivalent to the topology
    of pointwise convergence as $\bar D$ is compact.
    We must bring all samples from each measure $\gamma_n$
    to the space $\operatorname{Lip}(\bar D)$ for the large deviations principle to make
    sense.
    For each $n\in\mathbb N$, define the map $K_n:\Omega\to\operatorname{Lip}(\bar D)$
    as follows.
    First, for each $f\in\Omega$, define $\bar f$ to be the smallest Lipschitz
    extension of $f$ to $H$.
    Define each map $K_n$ by
    \[
        K_n(f):\bar D\to\mathbb R,\, x\mapsto \frac1n \bar f(nx).
    \]
    Finally, let $\lambda$ denote the unique translation-invariant measure
    on $H$ for which
    \[
       \{a_1\g_1+\dots+a_d\g_d:a_1,\dots,a_d\in [0,1]\}\subset H
    \]
    has measure one.
\end{definition}

\begin{theorem}
    \label{thm:ldp_new}
    Let $(D,b)$ denote a boundary profile with approximation $((D_n,b_n))_{n\in\mathbb N}$
    and associated measures $\gamma_n:=\gamma_{D_n}(\cdot,b_n)$.
    Write $\gamma_n^*$ for the pushforward of $\gamma_n$
    along $K_n$.
    Then the sequence
    of measures $(\gamma_n^*)_{n\in\mathbb N}$ satisfies a large deviations principle
    in the topological space $\operatorname{Lip}(\bar D)$
    with speed $n^d$ and rate function
    \[
        I(f):=-P(D,b)+
        \begin{cases}
            \int_D \sigma(\nabla f)d\lambda &\text{if $f|_{\partial D}=b$,}\\
            \infty &\text{otherwise,}
        \end{cases}
    \]
    where $P(D,b)$ is called the \emph{pressure}
    of the boundary profile $(D,b)$,
    defined to be the unique constant such that the minimum of $I$
    is $0$,
    and equal to
    \[
        P(D,b)=\lim_{n\to\infty}-\frac1{n^d}\log|\Omega(D_n,b_n)|.
    \]
\end{theorem}

We shall continue using the definitions of $I$ and $P(D,b)$ in the sequel.
The large deviations principle was proven in a much more general
setting by Sheffield in~\cite{SHEFFIELD}.
The large deviations principle with boundary conditions is stated in Subsection~7.3.2.
The large deviations principle in~\cite{SHEFFIELD}
does not only address the macroscopic profile of each height functions but also
its ``local statistics'' within macroscopic regions,
something we are not concerned with here.
In~\cite{SHEFFIELD} it is required that the boundary profile is ``not taut''.
This requirement is however only necessary to understand the local statistics,
and may be omitted when one is interested in the macroscopic profile only.
For a more recent and elementary proof of the large deviations principle
for the macroscopic profile only, we refer to Theorem~2.17 in~\cite{MINIMAL}.

\subsection{The variational principle}

The variational principle is a direct corollary of the large deviations principle.
Note that the set of minimisers of $I$ in Theorem~\ref{thm:ldp_new}
is exactly the set of minimisers of
\[
    \int_D\sigma(\nabla f)d\lambda
\]
over all functions $f\in\operatorname{Lip}(\bar D)$
which restrict to $b$ on $\partial D$.

\begin{theorem}
    \label{thm:vp_new}
    Assume the setting of Theorem~\ref{thm:ldp_new}.
    Let $A$ denote an open neighbourhood of $\{I=0\}\subset \operatorname{Lip}(\bar D)$.
    Then
    $\gamma_n^*(A)\to 1$ as $n\to\infty$.
    In particular, if $\sigma$ is strictly convex,
    then $I$ has a unique minimiser $f^*\in \operatorname{Lip}(\bar D)$,
    and in that case $\gamma_n^*(A)\to 1$ as $n\to\infty$
    for any open neighbourhood $A$ of $f^*$.
\end{theorem}

\section{Strict convexity of the surface tension}
\label{sec:strict}

In~\cite{COHN}, the authors find
an explicit formula for $\sigma$
for the case $d=2$ by appealing to the integrable nature of the
model.
A direct corollary is that the surface tension is strictly convex.
In~\cite{SHEFFIELD}, Sheffield proves that the
surface tension related to \emph{any} simply attractive
model is strictly convex.
In particular, this implies the following theorem.

\begin{theorem}
  \label{thm:strictly_convex}
 For any $d\geq 2$, the surface tension
  is strictly convex on the interior of $\mathcal S$.
\end{theorem}

The proof of Sheffield relies crucially on cluster swapping.
The purpose of the section is to give an
alternative proof of Theorem~\ref{thm:strictly_convex},
which is simpler than the proof in~\cite{SHEFFIELD}
due to the special nature of the cluster swap in the particular setting of this article.

\subsection{The specific entropy}
\label{subsec:specEntropy}
First, we give an alternative characterisation
of $\sigma(s)$ in terms of the shift-invariant gradient
Gibbs measure of slope $s$ whose existence is guaranteed
by Theorem~\ref{thm:existence_big_Gibbs_measures}.
For this, we require the notions
of \emph{entropy} and \emph{specific entropy}.
Let $(X,\mathcal X)$ be an arbitrary measurable space
endowed with
a probability measure $\mu$ and a nonzero finite measure $\nu$.
Then the \emph{relative entropy of $\mu$ with respect to $\nu$},
denoted $\mathcal H(\mu,\nu)$, is defined by
\[\mathcal H(\mu,\nu):=\begin{cases}
  \nu(h\log h)=\mu(\log h)&\text{if $\mu\ll\nu$ and $h=d\mu/d\nu$,}\\
  \infty&\text{if $\mu\not\ll\nu$.}
\end{cases}\]
If $\mathcal A$ is a sub-$\sigma$-algebra of $\mathcal X$,
then write $\mathcal H_\mathcal A(\mu,\nu)$
for $\mathcal H(\mu|_\mathcal A,\nu|_\mathcal A)$.
It is well-known that $\mu$ minimises $\mathcal H(\cdot,\nu)$
over all probability measures
if and only if $\mu$ is the normalised version of $\nu$,
in which case $\mathcal H(\mu,\nu)=-\log \nu(X)$.
Also, if $\nu$ is a counting measure,
then $h\leq 1$, and in that case $\mathcal H(\mu,\nu)\leq 0$.

If $R$ is a finite subset of $X^d$,
then write $D_R:\Omega\to\mathbb Z^{R\times R}$
for the map satisfying
\[(D_Rf)(\mathbf x,\mathbf y)=f(\mathbf y)-f(\mathbf x)\]
for every $f\in \Omega$, $\mathbf x,\mathbf y\in R$.
Call $D_R$ the \emph{differences map}.
Note that $\operatorname{Im}D_R$ is finite, and that
$\mathcal F_R^\nabla=\sigma(D_R)$.
Stronger: $D_R$ may be seen as a bijection
from $\mathcal F_R^\nabla$ to the powerset of $\operatorname{Im}D_R$.
Write $\lambda^R$ for the pullback of the counting measure on $\operatorname{Im}D_R$
along the map $D_R$---$\lambda^R$ is a measure on $(\Omega,\mathcal F_R^\nabla)$
of size $\lambda^R(\Omega)=|\operatorname{Im}D_R|\in\mathbb Z_{>0}$.

\begin{remark}
  If $R$ is connected and $f\in\Omega$, then
  the values of $D_Rf$ can be recovered from the values of $\nabla f$
  on the edges of $E^d$ which are contained in $R$, by integrating $\nabla f$ along the appropriate paths
  through $R$.
  Thus, for connected sets $R\subset X^d$, one should think of
  the map $D_R$ as an alternative for the map
  \[f\mapsto (\nabla f)|_{E^d\cap(R\times R)}\]
  in the above construction.
  This also implies that $|\operatorname{Im}D_R|\leq 2^{|R|}$ whenever $R$ is connected.
\end{remark}

Let $\mu\in\mathcal P(\Omega,\mathcal F^\nabla)$ and $R\subset X^d$ finite.
Then the \emph{entropy of $\mu$ in $R$}, denoted $\mathcal H_R(\mu)$,
is defined by
\[\mathcal H_R(\mu)
:=\mathcal H_{\mathcal F_R^\nabla}(\mu,\lambda^R)
=\sum_{x\in\operatorname{Im}D_R}\mu(D_Rf=x)\log\mu(D_Rf=x)\in [-\log |\operatorname{Im}D_R|,0].\]
The \emph{specific entropy} of $\mu$, denoted $\mathcal H(\mu)$,
is defined to be the limit
\[\mathcal H(\mu):=\lim_{n\to\infty}\frac1{|\Pi_n|}\mathcal H_{\Pi_n}(\mu)=\lim_{n\to\infty}\frac1{|\Pi_n|}\mathcal H_{\mathcal F_{\Pi_n}^\nabla}(\mu,\lambda^{\Pi_n})\]
whenever the sequence is convergent.
Otherwise simply replace the limit by the limit inferior
to obtain a well-defined limit.
It can in fact be shown that the sequence is always convergent, see for example~\cite[Chapter~2]{SHEFFIELD},
but we shall not rely on this fact.

\begin{theorem}
  \label{thm_specific_entropy_single_gradient_measure}
  Let $\mu$ denote a measure of Theorem~\ref{thm:existence_big_Gibbs_measures} of slope $s\in\mathcal S$.
   Then $\mathcal H(\mu)=\sigma(s)$.
\end{theorem}

\begin{proof}
  Let $\mu$ denote any gradient Gibbs measure for now.
  Write
  $h^R$ for the Radon-Nikodym derivative
  \[h^R := \frac{d\mu|_{\mathcal F_R^\nabla}}{d\lambda^R}\]
   for any finite $R\subset X^d$.
  Fix $R,S\subset X^d$ finite with $S\cup\partial S\subset R$.
  As $\mu$ is Gibbs, we know that $\mu$ is uniformly random in $\Omega(S,f)$
  whenever $\mu$ is conditioned on the values of $f$ on $S^\complement$.
  This implies immediately that
  \[h^R=\frac1{|\Omega(S,f)|}h^{R\smallsetminus S}.\]
  The function $|\Omega(S,\cdot )|$ is $\mathcal F_{\partial S}^\nabla$-measurable
  as the model is Markov,
  and consequently $h^R$ is $\mathcal F_{R\smallsetminus S}^\nabla$-measurable.

  Let now $\mu$ be a measure of Theorem~\ref{thm:existence_big_Gibbs_measures} of slope $s\in\mathcal S$,
  and pick $n\in\mathbb N$.
  Then
  \begin{align*}
    (2n)^{-d}\mathcal H_{\Pi_n}(\mu)
    &
    =(2n)^{-d}\mu(\log h^{\Pi_n})
    =(2n)^{-d}\mu(\log h^{\Pi_n\smallsetminus \Pi_{n-1}}-\log |\Omega(\Pi_{n-1},f)|)
    \\&=
    \numberthis
    \label{eq_this_is_the_limit}(2n)^{-d}\mathcal H_{\Pi_n\smallsetminus\Pi_{n-1}}(\mu)+\mu(-(2n)^{-d}\log|\Omega(\Pi_{n-1},f)|).
  \end{align*}
The first term in~\eqref{eq_this_is_the_limit} vanishes as $n\to\infty$
because $\Pi_n\smallsetminus\Pi_{n-1}$ is connected as a subset of $(X^d,E^d)$:
\[
  |\mathcal H_{\Pi_n\smallsetminus\Pi_{n-1}}(\mu)|\leq \log |\operatorname{Im}D_{\Pi_n\smallsetminus\Pi_{n-1}}|\leq |{\Pi_n\smallsetminus\Pi_{n-1}}|\log 2=O(n^{d-1}).\]
The term within the expectation in (\ref{eq_this_is_the_limit}) converges to $\sigma(s)$ pointwise
by Proposition~\ref{propo:limitFlatness} and
Theorem~\ref{thm:conv_to_sigma_stable}.
We may apply the dominated convergence theorem
because the expression within the expectation is always absolutely bounded
by $\log 2$, since
$0\leq \log |\Omega(\Pi_{n-1},f)|\leq |\Pi_{n-1}|\log 2 \leq (2n)^d\log 2$.
\end{proof}

Before proceeding, let us quote an important result from the literature.

\begin{theorem}
 \label{thm:no_gibbs_no_sigma}
  Let $\mu\in\mathcal P_\Theta(\Omega,\mathcal F^\nabla)$
  denote a measure which satisfies the concentration of~\eqref{eq:flat_conc}
  for some $s\in\mathcal S$,
  but which is not a Gibbs measure.
  Then $\mathcal H(\mu)>\sigma(s)$.
\end{theorem}

\begin{proof}[Proof overview]
  Write $\mu^{n,g}$ for the measure $\mu$
  conditioned on $f(\x)=g(\x)$
  for all $\x\in\Pi_n\smallsetminus\Pi_{n-1}$.
  Then
  \[
    \mathcal H_{\Pi_n}(\mu)=\mathcal H_{\Pi_n\smallsetminus\Pi_{n-1}}(\mu)
    +\int \mathcal H_{\Pi_n}(\mu^{n,g}) d\mu(g).
  \]
  This is the same decomposition as in the proof of the previous theorem.
  For fixed $g$, the integrand in this display is clearly minimised
  if $\mu$ is a Gibbs measure,
  because $\mu^{n,g}$ is then uniformly random in all extensions
  of $g|_{\Pi_n\smallsetminus\Pi_{n-1}}$ to $\Pi_n$.
  This proves that $\mathcal H(\mu)\geq\sigma(s)$
    whenever $\mu$ is concentrated as in~\eqref{eq:flat_conc};
  the difficulty is in proving the strict inequality whenever
  $\mu$ is not Gibbs.
  If $\mu$ is not Gibbs, then the integral in the display
  will for some $n$ be strictly larger than if $\mu$
  were Gibbs, but it is nontrivial to demonstrate that this difference
  survives the normalization by $|\Pi_n|$
  in the definition of $\mathcal H(\mu)$.
  This follows from a standard superadditivity argument,
  see Lemma~4.1 in~\cite{SELF_VP}
  or Theorem~15.37 in~\cite{GEORGII}.
  See Theorem~2.5.2 in~\cite{SHEFFIELD}
  for a proof of the current theorem in full detail.
\end{proof}

\subsection{The product setting}

For the double dimer model, the cluster swap, and the level set decomposition
developed in this article, it is essential to work in the product setting.
We shall introduce some straightforward technical machinery before proceeding;
essentially we must adapt the constructions and results from Section~\ref{sec:gibbs}
and from the previous subsection to
the product setting.
Write
\begin{alignat*}{5}
  &\mathcal F^{2\nabla}&&:=\mathcal F^{\nabla}\times\mathcal F^{\nabla},   \qquad   && \gamma_R^2&&:=\gamma_R\times\gamma_R
  ,\qquad &&\lambda^R_2:=\lambda^R\times\lambda^R, \\
  &\mathcal F^{2\nabla}_R&&:=\mathcal F^{\nabla}_R\times\mathcal F^{\nabla}_R,  \qquad && \gamma_R^{2\nabla}&&:=\gamma_R^{\nabla}\times\gamma_R^{\nabla}.
\end{alignat*}
Let $\mathcal P(\Omega^2,\mathcal F^{2\nabla})$ denote the collection of probability measures on $(\Omega^2,\mathcal F^{2\nabla})$;
such measures are called \emph{double gradient measures}.
If $\mu\in\mathcal P(\Omega^2,\mathcal F^{2\nabla})$ then we shall by default write
 $(f_1,f_2)$ for the pair of random height functions, and $g:=f_1-f_2$ for the random
 difference.
 Write $\mathcal P_\Theta(\Omega^2,\mathcal F^{2\nabla})$
 for the collection of shift-invariant measures $\mu\in\mathcal P(\Omega^2,\mathcal F^{2\nabla})$;
 the measure $\mu$
 is called \emph{shift-invariant} if $\mu(\tilde\theta A\times \tilde\theta B)=\mu(A\times B)$
 for every $\theta\in\Theta$ and $A,B\in\mathcal F^\nabla$.
 This is equivalent to requiring that $(\nabla f_1,\nabla f_2)$ and $(\theta \nabla f_1,\theta\nabla f_2)$
 have the same law in $\mu$ for each shift $\theta\in\Theta$.

The kernel $\gamma^2_R=\gamma_R\times\gamma_R$ is simply the kernel
from $(\Omega^2,\mathcal F_{R^\complement}^2)$
to
$(\Omega^2,\mathcal F^2)$
with the property that the probability measure $(\gamma_R\times\gamma_R)(\cdot,(f_1,f_2))$
is uniformly random in the set $\Omega(R,f_1)\times\Omega(R,f_2)$,
and it restricts naturally to the kernel $\gamma^{2\nabla}_R=\gamma_R^\nabla\times\gamma_R^\nabla$---this is a probability kernel from $(\Omega^2,\mathcal F^{2\nabla}_{R^\complement})$
to $(\Omega^2,\mathcal F^{2\nabla})$.
A double gradient measure $\mu$ is called a \emph{double gradient Gibbs measure}
if it satisfies, for every finite $R\subset X^d$, the DLR equation
\[\mu=\mu\gamma_R^{2\nabla}.\]
If $\mu$ is the product of two gradient Gibbs measures $\mu_1$ and $\mu_2$,
then $\mu$ is also Gibbs as
\[
\mu
=\mu_1\times\mu_2=(\mu_1\gamma_R^\nabla)\times(\mu_2\gamma_R^\nabla)
=(\mu_1\times\mu_2)(\gamma_R^\nabla\times\gamma_R^\nabla)
=\mu\gamma_R^{2\nabla}.
\]

Now let $\mu\in\mathcal P(\Omega^2,\mathcal F^{2\nabla})$
and $R\subset X^d$ finite.
The \emph{entropy of $\mu$ in $R$},
denoted $\mathcal H_R^2(\mu)$,
is defined by
\begin{align*}
  \mathcal H_R^2(\mu)
  &:=\mathcal H_{\mathcal F_R^{2\nabla}}(\mu,\lambda^R_2).
\end{align*}
The \emph{specific entropy} of $\mu$, denoted $\mathcal H^2(\mu)$,
is defined to be the limit
\[\mathcal H^2(\mu)
:=\lim_{n\to\infty}\frac1{|\Pi_n|}\mathcal H_{\Pi_n}^2(\mu)
=\lim_{n\to\infty}\frac1{|\Pi_n|}\mathcal H_{\mathcal F_{\Pi_n}^{2\nabla}}(\mu,\lambda^{\Pi_n}_2)\]
whenever the limit is convergent,
and the limit inferior otherwise.

The direct generalisation of Theorems~\ref{thm_specific_entropy_single_gradient_measure}
and~\ref{thm:no_gibbs_no_sigma}
to the product setting reads as follows.

\begin{theorem}
  \label{thm:H_sigma_ineq_product}
  Let $\mu\in\mathcal P_\Theta(\Omega^2,\mathcal F^{2\nabla})$
  denote a shift-invariant product measure
  such that
  \begin{equation}
    \label{double_concentration}
    \lim_{n\to\infty}\frac1n\|(f_i-f_i(\0))|_{\Pi_n}-s_i|_{\Pi_n}\|_\infty=0
  \end{equation}
  almost surely for $i\in\{1,2\}$
  and for some fixed slopes $s_1,s_2\in\mathcal S$.
  Then $\mathcal H^2(\mu)\geq \sigma(s_1)+\sigma(s_2)$,
  with equality if and only if $\mu$ is a Gibbs measure.
\end{theorem}

\subsection{Proof overview}
Fix throughout this section two distinct Lipschitz slopes $s_1,s_2\in\mathcal S$
such that their average $s_a:=(s_1+s_2)/2$
lies in the interior of $\mathcal S$.
The ultimate goal of this section is to prove that
$2\sigma(s_a)<\sigma(s_1)+\sigma(s_2)$,
which implies Theorem~\ref{thm:strictly_convex}: that $\sigma$
is strictly convex on the interior of $\mathcal S$.

In the remainder of this section,
let $\mu_i$ denote the shift-invariant gradient Gibbs measure of Theorem~\ref{thm:existence_big_Gibbs_measures}
of slope $s_i$ for each $i\in\{1,2\}$,
and fix $\mu:=\mu_1\times\mu_2$.
Then $\mu$ is a shift-invariant double gradient Gibbs measure.
Moreover, $\mu$ has the concentration of~\eqref{double_concentration},
and therefore
Theorem~\ref{thm:H_sigma_ineq_product}
implies that $\mathcal H^2(\mu)=\sigma(s_1)+\sigma(s_2)$.

The sets $T(f_1)$ and $T(f_2)$, the graph $G_g=(V_g,E_g)$, the $g$-level sets,
the $g$-boundaries, and the directed graph $(\operatorname{LSD}(g),\nabla g)$ are all invariant
under adding constants to $f_1$ and $f_2$,
as each of them is characterised entirely by the gradients $\nabla f_1$, $\nabla f_2$, and  $\nabla g:=\nabla f_1-\nabla f_2$.
The gradient $\nabla g$ also determines $X_g^\pm(E)$
for any $g$-boundary $E$.

\begin{lemma}
  \label{lemma_lsd_right_infinite_structure}
  It is $\mu$-almost certain that $\operatorname{LSD}(g)$ contains a subgraph that
  is graph isomorphic to $\mathbb Z$.
  Moreover, every $g$-level set and every $g$-boundary
  involved in such a subgraph of $\operatorname{LSD}(g)$ is unbounded.
\end{lemma}

This lemma is essential in understanding the geometry of $\operatorname{LSD}(g)$.
It is expected that the difference function $g=f_1-f_2$ of a typical sample
from $\mu$
looks somewhat like the leftmost subfigure of Figure~\ref{fig:trifurcation}.

\begin{proof}[Proof of Lemma~\ref{lemma_lsd_right_infinite_structure}]
  As $s_g:=s_1-s_2\neq 0$, there exists an index $1\leq i\leq d+1$
  such that $s_g(\mathbf g_i)\neq 0$.
  Fix such an $i$,
  and write $\mathbf p$ for the $\mathbb Z$-indexed path $\mathbf p:=(\mathbf p_k)_{k\in\mathbb Z}:=(k\mathbf g_i)_{k\in\mathbb Z}$ through $(X^d,E^d)$.
  Write $\mathbf q_k$ for the $g$-level
  set containing $\mathbf p_k$ for each $k\in\mathbb Z$.
  For each $k\in\mathbb Z$ the vertices $\mathbf p_k$ and $\mathbf p_{k+1}$
  are either contained in the same $g$-level set, or in two distinct neighbouring
  $g$-level sets.
  We consider $\mathbf q:=(\mathbf q_k)_{k\in\mathbb Z}$ a walk through $\operatorname{LSD}(g)$ although $\mathbf q$
  is not a walk in the strict sense: it may visit the same $g$-level set multiple times in a row.
%  For simplicity we assume that $s_g(\mathbf g_i)>0$.
  Proposition~\ref{propo:limitFlatness} says
  that $\mu$-almost surely $g(\mathbf p_k)-g(\mathbf p_0) = k s_g(\mathbf g_i) +o(k)$
  as $k\to \infty$ or $k\to-\infty$.
  This implies that there is a well-defined and unique loop-erased bi-infinite version
  of the path $\mathbf q$ up to indexation, which is the desired
  $\mathbb Z$-isomorphic subgraph of $\operatorname{LSD}(g)$.
  This proves the first part of the lemma.

  Focus on the second statement, which is deterministic in nature.
  Fix a $g$-boundary $E\subset E^d$ that is an edge of a
  subgraph of $\operatorname{LSD}(g)$ that is isomorphic $\mathbb Z$.
  Then removing $E$
  from $\operatorname{LSD}(g)$ disconnects $\operatorname{LSD}(g)$
  and separates the graph into two infinite components.
  In particular, this implies that
  the graph $(X^d,E^d\smallsetminus E)$ consists of
  two infinite connected components.
  If $E$ were finite, then one of the two connected components of
  $(X^d,E^d\smallsetminus E)$ had to be finite, and therefore we conclude that $E$ is infinite.
  The $g$-boundary $E$
  connects the two $g$-level sets $X_g^-(E)$ and $X_g^+(E)$ when considered an $\operatorname{LSD}(g)$-edge,
  and these must also be infinite as one of them contains the infinite set $\mathbf x^-_g(E)$
  and the other $\mathbf x^+_g(E)$.
  This proves the second statement of the lemma.
\end{proof}

% !TEX root =  ../main.tex

\begin{figure}
\begin{center}
\begin{tikzpicture}[x=1cm,y=1cm]

  \begin{scope}[shift={(0,0)}]
    \node[anchor=north] at (2.25,0) {$\mu=\mu_1\times\mu_2$};
      \clip (0,0) rectangle (4.5,5);
      \fill[Aquamarine!10] (0,0) rectangle (4.5,5);
      \begin{scope}[shift={(0,3.8)}]
      \fill[Aquamarine!25] plot [smooth cycle] coordinates{
        (4.5,-5)
        (4.8,0.2)
        (3.2,0.4)
        (2,1.2)
        (1,1.2)
        (-0.5,1.4)
        (0,-5)
        };
      \end{scope}
      \begin{scope}[shift={(0,3)}]
      \fill[Aquamarine!40] plot [smooth cycle] coordinates{
        (4.5,-5)
        (4.8,0.2)
        (3,0.6)
        (2,1.2)
        (1,1.2)
        (-0.5,1.4)
        (0,-5)
        };
      \end{scope}
      \begin{scope}[shift={(0,2.2)}]
      \fill[Aquamarine!55] plot [smooth cycle] coordinates{
        (4.5,-5)
        (4.8,0.2)
        (3,0.6)
        (2,1.2)
        (1,1)
        (-0.5,1.4)
        (0,-5)
        };
      \end{scope}
      \begin{scope}[shift={(0,1.4)}]
      \fill[Aquamarine!70] plot [smooth cycle] coordinates{
        (4.5,-5)
        (4.8,0.2)
        (3,0.6)
        (1,1.2)
        (-0.5,1.4)
        (0,-5)
        };
      \end{scope}
      \begin{scope}[shift={(0,0.6)}]
      \fill[Aquamarine!85] plot [smooth cycle] coordinates{
        (4.5,-5)
        (4.8,0.2)
        (3,0.6)
        (2,1.2)
        (1,1.2)
        (-0.5,1.4)
        (0,-5)
        };
      \end{scope}
      \begin{scope}[shift={(0,-0.2)}]
      \fill[Aquamarine] plot [smooth cycle] coordinates{
        (4.5,-5)
        (4.8,0.2)
        (3,0.6)
        (2,1)
        (1,0.8)
        (-0.5,1.4)
        (0,-5)
        };
      \end{scope}
  \end{scope}

  \begin{scope}[shift={(5,0)}]
    \node[anchor=north] at (2.25,0) {$\hat\mu$};
      \clip (0,0) rectangle (4.5,5);
      \fill[Aquamarine!30] (0,0) rectangle (4.5,5);
      \begin{scope}[shift={(0,3.8)}]
      \fill[Aquamarine!55] plot [smooth cycle] coordinates{
        (4.5,-5)
        (4.8,0.2)
        (3.2,0.4)
        (2,1.2)
        (1,1.2)
        (-0.5,1.4)
        (0,-5)
        };
      \end{scope}
      \begin{scope}[shift={(0,3)}]
      \fill[Aquamarine!30] plot [smooth cycle] coordinates{
        (4.5,-5)
        (4.8,0.2)
        (3,0.6)
        (2,1.2)
        (1,1.2)
        (-0.5,1.4)
        (0,-5)
        };
      \end{scope}
      \begin{scope}[shift={(0,2.2)}]
      \fill[Aquamarine!55] plot [smooth cycle] coordinates{
        (4.5,-5)
        (4.8,0.2)
        (3,0.6)
        (2,1.2)
        (1,1)
        (-0.5,1.4)
        (0,-5)
        };
      \end{scope}
      \begin{scope}[shift={(0,1.4)}]
      \fill[Aquamarine!80] plot [smooth cycle] coordinates{
        (4.5,-5)
        (4.8,0.2)
        (3,0.6)
        (1,1.2)
        (-0.5,1.4)
        (0,-5)
        };
      \end{scope}
      \begin{scope}[shift={(0,0.6)}]
      \fill[Aquamarine!55] plot [smooth cycle] coordinates{
        (4.5,-5)
        (4.8,0.2)
        (3,0.6)
        (2,1.2)
        (1,1.2)
        (-0.5,1.4)
        (0,-5)
        };
      \end{scope}
      \begin{scope}[shift={(0,-0.2)}]
      \fill[Aquamarine!30] plot [smooth cycle] coordinates{
        (4.5,-5)
        (4.8,0.2)
        (3,0.6)
        (2,1)
        (1,0.8)
        (-0.5,1.4)
        (0,-5)
        };
      \end{scope}
      \draw[very thick] (1,1.5) rectangle (3.5,4);
      \node[anchor=south west] at (1,1.5) {$\Pi_n$};
  \end{scope}

  \begin{scope}[shift={(10,0)}]
    \node[anchor=north] at (2.25,0) {$\hat\mu\gamma_{\Pi_m}^{2\nabla}$};
      \clip (0,0) rectangle (4.5,5);
      \fill[Aquamarine!30] (0,0) rectangle (4.5,5);
      \begin{scope}[shift={(0,3.8)}]
      \fill[Aquamarine!55] plot [smooth cycle] coordinates{
        (4.5,-5)
        (4.8,0.2)
        (3.2,0.4)
        (2,1.2)
        (1,1.2)
        (-0.5,1.4)
        (0,-5)
        };
      \end{scope}
      \begin{scope}[shift={(0,3)}]
      \fill[Aquamarine!30] plot [smooth cycle] coordinates{
        (4.5,-5)
        (4.8,0.2)
        (3,0.6)
        (2,1.2)
        (1,1.2)
        (-0.5,1.4)
        (0,-5)
        };
      \end{scope}
      \begin{scope}[shift={(0,2.2)}]
      \fill[Aquamarine!55] plot [smooth cycle] coordinates{
        (4.5,-5)
        (4.8,0.2)
        (3,0.6)
        (2,1.2)
        (1,1)
        (-0.5,1.4)
        (0,-5)
        };
      \end{scope}
      \begin{scope}[shift={(0,1.4)}]
      \fill[Aquamarine!80] plot [smooth cycle] coordinates{
        (4.5,-5)
        (4.8,0.2)
        (3,0.6)
        (1,1.2)
        (-0.5,1.4)
        (0,-5)
        };
      \end{scope}
      \begin{scope}[shift={(0,0.6)}]
      \fill[Aquamarine!55] plot [smooth cycle] coordinates{
        (4.5,-5)
        (4.8,0.2)
        (3,0.6)
        (2,1.2)
        (1,1.2)
        (-0.5,1.4)
        (0,-5)
        };
      \end{scope}
      \begin{scope}[shift={(0,-0.2)}]
      \fill[Aquamarine!30] plot [smooth cycle] coordinates{
        (4.5,-5)
        (4.8,0.2)
        (3,0.6)
        (2,1)
        (1,0.8)
        (-0.5,1.4)
        (0,-5)
        };
      \end{scope}
      \fill[Aquamarine!55] plot [smooth cycle] coordinates{
        (0.8,1.5)
        (0.8,3)
        (0.9,4.3)
        (3.7,3.9)
        (3.5,0.8)
        };
      \draw[very thick] (1,1.5) rectangle (3.5,4);
      \node[anchor=south west] at (1,1.5) {$\Pi_n$};
      \draw[very thick] (0.2,0.7) rectangle (4.3,4.8);
      \node[anchor=north west] at (0.2,0.7) {$\Pi_m$};
  \end{scope}

\end{tikzpicture}

\caption[A trifurcation box]{
Difference samples from three different measures;
$\Pi_m$ is a trifurcation box.
}

\label{fig:trifurcation}

\end{center}
\end{figure}
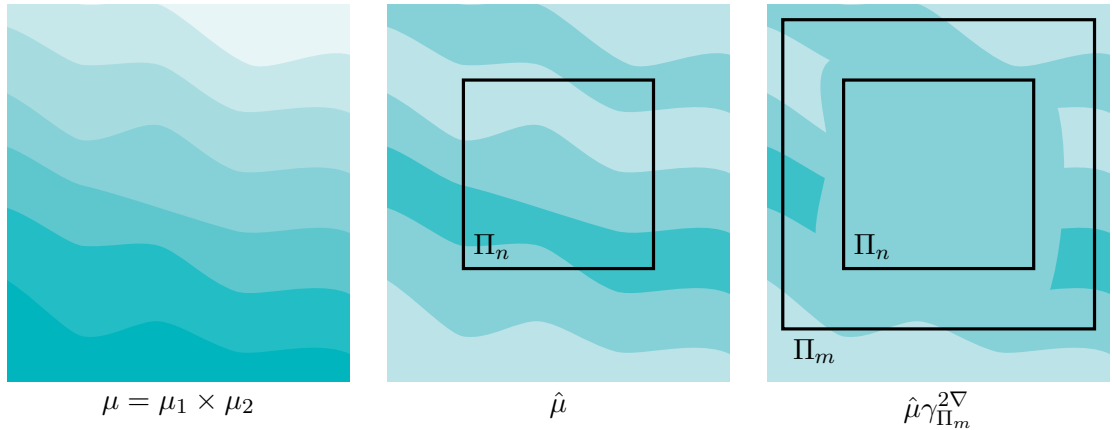

We now give an overview of the remainder of the proof.
The key idea is to construct a new shift-invariant double gradient measure
$\hat\mu\in\mathcal P_\Theta(\Omega^2,\mathcal F^{2\nabla})$.
Write $\hat f_1$, $\hat f_2$ and $\hat g:=\hat f_1-\hat f_2$
for the random functions in $\hat \mu$.
To sample from $\hat\mu$, first draw a pair $(f_1,f_2)$
from the original measure $\mu=\mu_1\times\mu_2$.
Then obtain $(\hat f_1,\hat f_2)$ from $(f_1,f_2)$
by flipping a fair coin for every $g$-boundary in order to determine
wether or not to reverse the orientation of that $g$-boundary.
In other words, we rerandomise the orientation of each $g$-boundary.
In the measure $\hat\mu$,
the orientations of the edges in the graph $(\operatorname{LSD}(\hat g),\nabla\hat g)$
are thus uniformly random and independent of all other structure that is present.
First, we show that
the resampling operation does not affect the specific entropy, that is,
\[\mathcal H^2(\hat \mu)=\mathcal H^2(\mu)=\sigma(s_1)+\sigma(s_2).\]
Second, we prove that for $i\in\{1,2\}$
we $\hat \mu$-almost surely have
\begin{equation}
  \label{eq:desired_concentration_hat_mu_marginals}
  \lim_{n\to\infty}\frac 1n \| (\hat f_i-\hat f_i(\mathbf 0))|_{\Pi_n}-s_a|_{\Pi_n}\|_\infty = 0.
\end{equation}
Note that the concentration of the gradient of either function
is around the average slope $s_a$.
Third, we prove that $\hat \mu$ is not Gibbs.
Theorem~\ref{thm:H_sigma_ineq_product} therefore implies that $\mathcal H^2(\hat\mu)>2\sigma(s_a)$,
the desired result.

Let us elaborate on the third step, before proceeding.
Suppose that $\hat\mu$ is Gibbs, in order to derive a contradiction.
A \emph{trifurcation box} of $\hat g$
is a finite subset of $X^d$ of the form $R=\theta\Pi_n$
such that, for some infinite $\hat g$-level set $X\subset X^d$,
 removing $R$ from $X$ means breaking $X$ into
at least three infinite components.
We show that
$\hat g$ has a trifurcation box
with positive probability in the measure
$\hat\mu\gamma_{\Pi_m}^{2\nabla}$
for $m$ sufficiently large:
see the middle and rightmost subfigures in Figure~\ref{fig:trifurcation}.
If $\hat\mu$ were Gibbs then $\hat\mu=\hat\mu\gamma_{\Pi_m}^{2\nabla}$,
and therefore a sample $\hat g$ from $\hat\mu$
has a trifurcation box with positive probability.
Trifurcation boxes do almost surely not occur in shift-invariant
measures, by a simple geometrical argument described by Burton and Keane in their celebrated paper \cite{BURTON}.
This proves that $\hat \mu$ is not Gibbs.

\subsection{Detailed proof}

\begin{lemma}
  The specific entropy of $\hat\mu$ equals the specific entropy of $\mu$.
\end{lemma}

\begin{proof}
  We prove the stronger statement that $\mathcal H_{\Pi_n}^2(\hat\mu)=\mathcal H_{\Pi_n}^2(\mu)+O(n^{d-1})$ as $n\to\infty$.
  Fix $n\in\mathbb N$ large.
  The measure $\mu$ is Gibbs and therefore satisfies the DLR equation
  \[\mu=\mu\gamma_{\Pi_n}^{2\nabla}.\]
  This implies in particular that the
  distribution of
  a sample $(f_1,f_2)$ from $\mu$ is invariant
  under subsequently rerandomising the orientation of each $g$-boundary
  that is contained in $E^d(\Pi_n)$.
  As this is true for all $n$,
  we derive that the distribution of the sample is in fact invariant
  under rerandomising the orientation of any finite $g$-boundary.
  Thus, to sample $(\hat f_1,\hat f_2)$ from $\hat\mu$,
  one may first sample a pair $(f_1,f_2)$ from $\mu$,
  then rerandomise the orientation of only the $g$-boundaries
  which are infinite.
  Write $\tilde\mu$ for the corresponding coupling of $\mu$ and $\hat\mu$,
  and write $(f_1,f_2,\hat f_1,\hat f_2)$ for the random
  $4$-tuple in $\tilde\mu$.
%
%  We now construct a coupling $\tilde\mu$ between $\mu$ and $\hat\mu$.
%  To sample a $4$-tuple
%  $(f_1,f_2,\hat f_1,\hat f_2)$ from $\tilde\mu$, first
%  sample $(f_1,f_2)$ from $\mu$,
%  then obtain $(\hat f_1,\hat f_2)$ from
%  $(f_1,f_2)$
%  by rerandomising
%  all the $g$-boundaries that are not contained in $E^d(\Pi_{n-1})$.
%  The previous observation guarantees that the pair $(\hat f_1,\hat f_2)$
%  has the correct distribution.
%  The measure $\tilde\mu$ is a measure on the measurable space $(\Omega,\mathcal F^\nabla)^4$.

  Write $\mathcal H_{A,B}^2(\tilde\mu)$
  for $\mathcal H_{\mathcal F_A^{2\nabla}\times\mathcal F_B^{2\nabla}}(\tilde\mu,\lambda^A_2\times\lambda^B_2)$
  for $A,B\subset X^d$ finite and claim that
  \begin{enumerate}
    \item $\mathcal H_{\Pi_n}^2(\mu)=\mathcal H_{\Pi_n,\Pi_n\smallsetminus \Pi_{n-1}}^2(\tilde\mu)+O(n^{d-1})$,
    \item $\mathcal H_{\Pi_n}^2(\hat\mu)=\mathcal H_{\Pi_n\smallsetminus \Pi_{n-1},\Pi_n}^2(\tilde\mu)+O(n^{d-1})$,
    \item $\mathcal H_{\Pi_n,\Pi_n\smallsetminus \Pi_{n-1}}^2(\tilde\mu)=
    \mathcal H_{\Pi_n,\Pi_n}^2(\tilde\mu)
    =\mathcal H_{\Pi_n\smallsetminus \Pi_{n-1},\Pi_n}^2(\tilde\mu)$.
  \end{enumerate}
  It is clear that these three claims jointly imply the lemma.

  Focus on the first claim,
  and consider thus the measures
  \[
    \mu|_{\mathcal F_{\Pi_n}^{2\nabla}},\qquad
    \tilde\mu|_{\mathcal F_{\Pi_n}^{2\nabla}\times\mathcal F_{\Pi_n\smallsetminus \Pi_{n-1}}^{2\nabla}}.
  \]
  The first claim is intuitive:
  the restriction of $\mu$ records the values of $D_{\Pi_n}(f_1)$ and $D_{\Pi_n}(f_2)$,
  and the restriction of $\tilde\mu$ records also
  the values of $D_{\Pi_n\smallsetminus \Pi_{n-1}}(\hat f_1)$
  and $D_{\Pi_n\smallsetminus \Pi_{n-1}}(\hat f_2)$.
  Informally, the extra information that
  the restriction of $\tilde\mu$ records
  is of order $n^{d-1}$, because
    $\log |\operatorname{Im}D_{\Pi_n\smallsetminus \Pi_{n-1}}|^2=O(n^{d-1})$.
    We now formalise this idea.
  For $x\in (\operatorname{Im}D_{\Pi_n})^2$,
  we write $\hat\mu^x$ for the measure $\tilde\mu$ conditioned on the event $\{(D_{\Pi_n}f_1,D_{\Pi_n}f_2)=x\}$
  and projected onto the product of the third and fourth component
  of the product measurable space $(\Omega,\mathcal F^\nabla)^4$.
  Then a standard entropy calculation implies that
    \[
    \mathcal H_{\Pi_n,\Pi_n\smallsetminus \Pi_{n-1}}^2(\tilde\mu)
    =
    \mathcal H_{\Pi_n}^2(\mu)
    +\int
    \mathcal H^2_{\Pi_n\smallsetminus \Pi_{n-1}}(\hat\mu^x)
      d\mu((D_{\Pi_n}f_1,D_{\Pi_n}f_2)=x),
    \]
    see Theorems~D.3 and~D.13 of~\cite{DEMBO}
    or Lemma~2.1.3 of~\cite{SHEFFIELD}.
    As in the proof of Theorem~\ref{thm_specific_entropy_single_gradient_measure},
    we have
     \[|\mathcal H^2_{\Pi_n\smallsetminus \Pi_{n-1}}(\hat\mu^x)|\leq\log|\operatorname{Im}D_{\Pi_n\smallsetminus \Pi_{n-1}}|^2=O(n^{d-1}).\]
    This proves the first claim.
    The second claim follows by identical reasoning.

     Focus on the third claim, in particular on the equality on the left---the
    equality on the right shall follow by the same arguments.
    For the equality on the left it suffices to demonstrate that,
    with $\tilde\mu$-probability one,
    the tuple
    \begin{equation}
      \label{eq:oneoneone}
      (D_{\Pi_n}f_1,D_{\Pi_n}f_2,D_{\Pi_n}\hat f_1,D_{\Pi_n}\hat f_2)
    \end{equation}
    can be reconstructed almost surely from
  \begin{equation}
    \label{eq:twotwotwo}
    (D_{\Pi_n}f_1,D_{\Pi_n}f_2,D_{\Pi_n\smallsetminus\Pi_{n-1}}\hat f_1,D_{\Pi_n\smallsetminus\Pi_{n-1}}\hat f_2).
  \end{equation}
    We know that $(f_1,f_2)$ and $(\hat f_1,\hat f_2)$ differ
    by cluster boundary swaps,
    where  all boundaries which are swapped, are infinite.
  To recover~\eqref{eq:oneoneone}
  from~\eqref{eq:twotwotwo}, we must therefore understand
  wether or not each infinite boundary which intersects $\Pi_{n-1}$,
  should be swapped or not.
  However, as each such boundary is infinite,
  it must intersect $\Pi_n\smallsetminus\Pi_{n-1}$,
  and its orientation for $(\hat f_1,\hat f_2)$
  can therefore be read off from~\eqref{eq:twotwotwo}.
\end{proof}

\begin{lemma}
  \label{lem:desiredconcentrationyouknow}
  Equation~\ref{eq:desired_concentration_hat_mu_marginals}
  holds true
  $\hat\mu$-almost surely
  for $i\in\{1,2\}$.
\end{lemma}

\begin{proof}
  It suffices to prove that $\hat\mu$-almost surely
  \begin{align}
    \label{eq:sumtozeroasntoinfty}
    &\lim_{n\to\infty}\frac 1n \|
      (\hat f_1-\hat f_1(\mathbf 0))|_{\Pi_n}
     +(\hat f_2-\hat f_2(\mathbf 0))|_{\Pi_n}
     -2s_a|_{\Pi_n}
    \|_\infty = 0,\\
    \label{eq:sumtozeroasntoinfty2}
    &\lim_{n\to\infty}\frac 1n \|
    (\hat f_1-\hat f_1(\mathbf 0))|_{\Pi_n}
   -(\hat f_2-\hat f_2(\mathbf 0))|_{\Pi_n}
    \|_\infty = 0.
  \end{align}
  First focus on
   (\ref{eq:sumtozeroasntoinfty}).
   Recall the definition of $\mu$;
    Proposition~\ref{propo:limitFlatness} implies that almost surely
   \[
     \lim_{n\to\infty}\frac 1n \|
      (f_1-f_1(\mathbf 0))|_{\Pi_n}
     +(f_2-f_2(\mathbf 0))|_{\Pi_n}
     -(s_1+s_2)|_{\Pi_n}
    \|_\infty = 0.
   \]
    Recall  that
  the sum of two height functions is invariant under a cluster boundary swap,
  and that $2s_a=s_1+s_2$.
  The previous display therefore implies~\eqref{eq:sumtozeroasntoinfty}.
  For~\eqref{eq:sumtozeroasntoinfty2} we must show that $\hat\mu$-almost surely
  \begin{equation}
    \label{eq:abdc}
  \lim_{n\to\infty}\frac 1n \|
 (\hat g-\hat g(\mathbf 0))|_{\Pi_n}
  \|_\infty = 0
  \end{equation}
  Fix some $\mathbf x\in X^d$;
  we are interested in the distribution of $\hat g(\mathbf x)-\hat g(\mathbf 0)$.
  The distribution of $\hat g(\mathbf x)-\hat g(\0)$
  conditional on $\operatorname{LSD}(\hat g)$ is
  given by summing
  $d_{\operatorname{LSD}(\hat g)}(\0,\mathbf x)\leq
  d_{(X^d,E^d)}(\0,\mathbf x)$ fair coins each valued
  $\pm (d+1)$.
  This follows immediately from the definition of the measure $\hat\mu$.
  Application of the Azuma-Hoeffding inequality
  yields,
  for $a\geq 0$ and $\mathbf x\neq \mathbf 0$,
  \[
    \hat\mu(|\hat g(\mathbf x)-\hat g(\mathbf 0)|\geq (d+1)a)\leq 2\exp-\frac{a^2}{2d_{(X^d,E^d)}(\mathbf 0,\mathbf x)}.
  \]
    A union bound now implies~\eqref{eq:abdc}.
\end{proof}

\begin{lemma}
  \label{lem_notGibbs}
  The double gradient measure $\hat\mu$ is not Gibbs.
\end{lemma}

\begin{proof}
  Define the $\mathcal F^{2\nabla}$-measurable event
  \[I(n):=\left\{~\parbox{24em}{$\hat g$ takes the same value on three distinct infinite $\hat g$-level sets
  that all intersect $\Pi_n$, one of which contains $\mathbf 0$}~\right\}\subset \Omega^2,\]
  and claim that $\hat\mu(I(n))>0$ for $n\in\mathbb N$ sufficiently large.
  First,
  a cluster boundary swap
  leaves $\operatorname{LSD}(g)$ invariant,
  and therefore
    Lemma~\ref{lemma_lsd_right_infinite_structure}
  holds true for $g$ replaced with $\hat g$ and $\mu$ with $\hat\mu$.
  Therefore it is $\hat\mu$-almost certain that $\operatorname{LSD}(\hat g)$ contains a $\mathbb Z$-indexed self-avoiding walk
  $\mathbf p=(\mathbf p_k)_{k\in\mathbb Z}$.
  Let $\mathbf p$ be chosen deterministically in terms of $\operatorname{LSD}(\hat g)$,
  so that $(\hat g(\mathbf p_{k+1})-\hat g(\mathbf p_k))_{k\in\mathbb Z}$
  is a sequence of i.i.d.\ random variables each distributed uniformly
  in $\pm (d+1)$, independent of $\operatorname{LSD}(\hat g)$.
  In particular, the event
  $\{\hat g(\mathbf p_{k\pm 2})-\hat g(\mathbf p_k)=0\}$
  has probability $\frac14$ for each fixed $k$.
  As $\hat\mu$ is shift-invariant,
  we may choose $\mathbf p$ such that $\hat\mu(\mathbf 0\in\mathbf p_0)>0$.
  Choose $n\in\mathbb N$  sufficiently large
  such that, conditional on $\{\mathbf 0\in\mathbf p_0\}$,
  the set $\Pi_n$ intersects $\mathbf p_{\pm 2}$ with positive probability.
  Note that
$\hat g(\mathbf p_{\pm 2})-\hat g(\mathbf p_0)=0$ with probability $\frac14$
independently of the occurrence of both previous events,
and therefore the original event $I(n)$ has positive probability.
This is the claim.
Fix $n\in\mathbb N$ such that $\varepsilon:=\frac12\hat\mu(I(n))>0$.
See the middle display in Figure~\ref{fig:trifurcation}
for the level set decomposition of the difference function $\hat g$ corresponding to a sample from the event $I(n)$.

  If $h$ and $h'$ are real-valued functions defined
  on two disjoint subsets $A$ and $A'$ of $X^d$
  respectively,
  then write $hh'$ for the unique function
  on $A\cup A'$ which equals $h$ on $A$
  and $h'$ on $A'$.
Next, define for $m\geq n$ the $\mathcal F^{2\nabla}$-measurable event
\begin{align*}
L(m)&:=\left\{\text{the function $(\hat f_1-\hat f_1(\mathbf 0))|_{\Pi_n}(\hat f_2-\hat f_2(\mathbf 0))|_{\Pi_m^\complement}$ extends to a height function}\right\}
\\
&\phantom{:}=
\left\{\text{the function $(\hat f_1-\hat f_1(\mathbf 0))|_{\Pi_n}(\hat f_2-\hat f_2(\mathbf 0))|_{\Pi_m^\complement}$ is Lipschitz}\right\}
\subset \Omega^2,
\end{align*}
and claim that $\hat\mu(L(m))\to 1$ as $m\to\infty$.
Recall that~\eqref{eq:desired_concentration_hat_mu_marginals}
holds true for $i=2$ with $\hat\mu$-probability one.
Therefore it suffices to show that
the function
  \[
    (\hat f_1-\hat f_1(\mathbf 0))|_{\Pi_n}(\hat f_2-\hat f_2(\mathbf 0))|_{\Pi_m^\complement}
    \]
is Lipschitz for $m$ sufficiently large
whenever $\hat f_2$ satisfies~\eqref{eq:desired_concentration_hat_mu_marginals}.
  To see this, write $h_m^\pm$
  for the largest and smallest functions in
  $\Omega(\Pi_m,\hat f_2-\hat f_2(\0))$ respectively.
  In other words,
  the functions $h_m^\pm$ are the largest and smallest
  extensions of $(\hat f_2-\hat f_2(\0))|_{\Pi_m^\complement}$
  to $X^d$ which are height functions.
  Note that there exist constants $\alpha^\pm>0$
  such that $h_m^\pm(\0)=\pm\alpha^\pm m+o(m)$
  as $m\to\infty$ due to~\eqref{eq:desired_concentration_hat_mu_marginals}
  and because $s_a$ is in the \emph{interior} of
  the set of Lipschitz slopes.
  In particular,
  \[
  h_m^+|_{\Pi_n}\geq (\hat f_1-\hat f_1(\0))|_{\Pi_n},
    \qquad
  h_m^-|_{\Pi_n}\leq (\hat f_1-\hat f_1(\0))|_{\Pi_n}
  \]
  for $m$ sufficiently large.
This proves the claim.
Fix $m$ so large that $\hat\mu(L(m))\geq 1-\varepsilon$,
which implies that $\hat\mu(I(n)\cap L(m))\geq \varepsilon>0$.

Define for $\mathbf x\in X^d$ the $\mathcal F^{2\nabla}$-measurable event
\begin{align*}
  T(\mathbf x)&:=\left\{
~ \parbox{25em}{$(X^d,E^d\smallsetminus (V_{\hat g}\cup E^d(\Pi_m+\mathbf x)))$
has three infinite connected components that are contained in a single
$\hat g$-level set}
~  \right\}\\
&\phantom{:}=\left\{
~ \parbox{25em}{$(X^d,E^d\smallsetminus (V_{\hat g}\cup E^d(\Pi_m+\mathbf x)))$
has three infinite connected components that are contained in one
connected component of $(X^d,E^d\smallsetminus V_{\hat g})$}
~  \right\}
\subset \Omega^2;
\end{align*}
we shall first focus on $T:=T(\mathbf 0)$.
See the rightmost subfigure in Figure~\ref{fig:trifurcation}
for the level set decomposition of the difference function $\hat g$ corresponding to a sample from $T$.
We claim that $\Omega(\Pi_m,f_1)\times\Omega(\Pi_m,f_2)$
intersects $T$ whenever
$(f_1,f_2)\in I(n)\cap L(m)$.
Fix such a pair $(f_1,f_2)$ and assume, without loss of generality,
that $f_1(\mathbf 0)=f_2(\mathbf 0)=0$.
  Assert first that there exists a height function $f''$
  which equals $f_1$ on $\Pi_n$
  and which equals $f_2$ on $\Pi_m^\complement$,
  and which at any vertex $\x$ takes values between
  $f_1(\x)$ and $f_2(\x)$.
  The following construction works:
  write first $f'$ for a height function
  which extends $f_1|_{\Pi_n}f_2|_{\Pi_m^\complement}$
  to $X^d$,
  then define $f''$
  by
  \[f'':=(f'\vee(f_1\wedge f_2) ) \wedge (f_1 \vee f_2).\]
  This proves the assertion.
  Now $f''\in\Omega(\Pi_m,f_2)$ as $f''$ equals $f_2$ on $\Pi_m^\complement$.
  Moreover, as at each vertex the value of $f''$
  is in between the values of $f_1$ and $f_2$,
  we have $ \{f_1=f_2\}\subset\{f_1=f''\} \subset X^d$,
  and we also know that
  $\Pi_n\subset\{f_1=f''\} $.
  This implies that $(f_1,f'')\in T$,
  since
  the three connected components of $\{f_1=f_2\}$
  are contained in a single connected component of
  $\{f_1=f''\}$,
  and because $\{f_1=f_2\}\smallsetminus\Pi_m=\{f_1=f''\}\smallsetminus\Pi_m$,
  that is, removing $\Pi_m$ disconnects these three components again.

  Suppose that $\hat\mu$ is a Gibbs measure, in order to derive a contradiction.
  Then the event $T$ must occur with
  positive probability for $\hat\mu$,
  since $\hat\mu(I(n)\cap L(m))>0$,
  and since $T$ has positive probability for
  \[
    \gamma_{\Pi_m}^{2\nabla}(\cdot,(\hat f_1,\hat f_2))
  \]
  whenever $(\hat f_1,\hat f_2)\in I(n)\cap L(m)$.
  The argument of Burton and Keane~\cite{BURTON}
dictates that trifurcation boxes do almost surely not occur in shift-invariant
probability measures.
In other words, $\hat\mu(T)=0$. This is the desired contradiction.
  To see that $\hat\mu(T)=0$, observe that $\hat\mu(T(\x))$
  is independent of $\x$,
  implying that trifurcation boxes must occur with positive density
  whenever $\hat\mu(T)>0$.
  This is impossible due to geometrical constraints of the amenable graph $(X^d,E^d)$.
\end{proof}

This establishes Theorem~\ref{thm:strictly_convex}.

\section*{Acknowledgement}
The author is grateful to Nathana\"el Berestycki and James Norris for their enthusiastic supervision of the writing of this paper.
The author thanks Amir Dembo, Richard Kenyon, Scott Sheffield, and Martin Tassy
for helpful discussions.
Finally, the author would like to thank the anonymous referees for their insightful feedback.

The author was supported by the Department of
Pure Mathematics and Mathematical Statistics, University of Cambridge and the UK
Engineering and Physical Sciences Research Council grant EP/L016516/1.

\bibliographystyle{amsalpha}

\bibliography{references_proper/01_in_argument,references_proper/02_exact_models,references_proper/03_similar_surfaces,references_proper/04_classical_dimer_papers,references_proper/05_modern_dimers,rework}

\end{document}